%% file: long_time_linf_l2.tex
\documentclass[10pt,reqno,oneside, a4paper]{amsart}
\DeclareMathAlphabet{\mathcal}{OMS}{cmsy}{m}{n}
\usepackage[left=2.5cm,right=2.5cm, top=3cm,bottom=3.5cm]{geometry}
\usepackage{parts/macros}
\numberwithin{equation}{section}
\newtheorem{theorem}{Theorem}[section]
\newtheorem{lemma}[theorem]{Lemma}

\newtheorem{definition}[theorem]{Definition}
\let \citep \cite
\let \citet \cite
\title[Long time $L^{\infty}(L^2)$ a posteriori error estimates for parabolic problems]{Long time $L^{\infty}(L^2)$ a posteriori error estimates \\ for fully discrete parabolic problems}
\author{Oliver J. Sutton}
\address[O. J. Sutton]{Department of Mathematics and Statistics, University of Reading, Whiteknights, PO Box 220, Reading, RG6 6AX, UK}
\email{O.Sutton@reading.ac.uk}
\date{}
\begin{document}
\maketitle
\begin{abstract}
\input{parts/abstract}
\end{abstract}
\input{parts/mainbody}
\bibliographystyle{alphaabbr}
\bibliography{parts/references}
\input{parts/pagefigs}

\end{document}

%% file: parts/abstract.tex
Computable estimates for the error of finite element discretisations of parabolic problems in the $L^{\infty}(0,T; L^2(\domain))$ norm are developed, which exhibit constant effectivities (the ratio of the estimated error to the true error) with respect to the simulation time.
These estimates, which are of optimal order, represent a significant advantage for long-time simulations, and are derived using energy techniques based on elliptic reconstructions.
The effectivities of previous optimal order error estimates in this norm derived using energy techniques are shown numerically to grow either in proportion to the simulation duration or its square root, a key disadvantage compared with earlier estimators derived using parabolic duality arguments.
The new estimates form a continuous family, almost all of which are new, reproducing certain familiar energy-based estimates well suited for short-time simulations and not available through the parabolic duality framework.
For clarity, we demonstrate the technique applied to a linear parabolic problem discretised using standard conforming finite element methods in space coupled with backward Euler and Crank-Nicolson time discretisations, although it can be applied much more widely.

%% file: parts/mainbody.tex

\section{Introduction}
An unavoidable fact of life when simulating physical phenomena is that the approximate solutions produced by discrete schemes, such as finite element methods and their relations, will not perfectly match the true solutions of the model, resulting in some discretisation error.
A natural question this raises, and one into which a breathtaking quantity of work has been invested, is whether this discretisation error can be quantified for a given simulation, knowing only the problem data and computed approximate solution.
Such \emph{computable a posteriori error estimates} can then be used to determine how faithful the discrete solutions are to the model, thus providing an indication of the reliability of the simulation and opening the door to the development of rigorous adaptive algorithms.
Although the derivation of such error estimates for models based on systems of elliptic partial differential equations is by now rather mature (we refer to~\cite{Ainsworth:2000cv,Brenner:2008tq,Verfurth:uz}, for instance), interesting open questions still remain for models based on parabolic and hyperbolic partial differential equations.

For linear second order parabolic problems there are two particularly natural norms in which to measure the error, the $L^2(0,t; H^1(\domain))$ norm and the $L^{\infty}(0,t; L^2(\domain))$ norm (a description of the notation we use here for Sobolev spaces and their associated norms is given in Section~\ref{sec:modelProblem}; for further details on these and Bochner spaces see~\cite{Adams:2003wi}), both of which arise through conventional energy arguments.
Optimal order a posteriori estimates for the error in the $L^2(0,t; H^1(\domain))$ norm may be proven using direct energy arguments, as shown by~\cite{Picasso:1998iw} and~\cite{Chen:2004ck}.
Although the same arguments provide a bound for the (higher order) error in the $L^{\infty}(0,t; L^2(\domain))$ norm, the resulting estimators are in fact of suboptimal order.

Optimal order estimates for the $L^{\infty}(0,t; L^2(\domain))$ norm error were first proven by 
Eriksson and Johnson~
\cite{ERIKSSON:1991be,ERIKSSON:1995in} using duality techniques.
A significant recent breakthrough, however, was the introduction of the \emph{elliptic reconstruction} technique by 
Makridakis and Nochetto in~
\cite{Makridakis:2003ws}, allowing optimal order error estimates to be derived in the $L^{\infty}(0,t; L^2(\domain))$ norm via energy arguments by introducing `elliptic reconstructions' of the discrete solution.
The role these reconstructions play in the derivation of the error estimates may be seen as an a posteriori counterpart to the role of the Ritz projection in the derivation of optimal order a priori error estimates in the $L^{\infty}(0,t; L^2(\domain))$ norm, first deployed by 
Wheeler~
\cite{Wheeler:1973hd}.
This splits the error into an elliptic component, which is treated using a posteriori error estimates derived for the associated elliptic problem, and a parabolic component which satisfies a differential equation with data which may be controlled at optimal order.
In the fully (space and time) discrete setting, these a posteriori error bounds are typically proven by combining the elliptic reconstruction in space with a suitable piecewise polynomial time reconstruction.
An overview of this methodology is given by~\cite{Makridakis:2007ef}, and the combination of space and time reconstructions has shown itself to be very flexible in many settings.
For instance, optimal order $L^{\infty}(0,t; L^2(\domain))$ error estimates have been proven using this technique for standard conforming finite element methods, coupled with backward Euler~\citep{Lakkis:2006jk}, Crank-Nicolson~\citep{Lozinski:2009bt,Bansch:2012gu,Bansch:2013cu}, fractional $\theta$-step~\citep{Karakatsani:2012ci}, and general $hp$ discontinuous Galerkin~\citep{Georgoulis:2017wq} time-stepping schemes, to name but a few.

A problem afflicting the above $L^{\infty}(0,t; L^2(\domain))$ norm error estimates derived using energy techniques with elliptic reconstructions, however, is that their \emph{effectivity}, the ratio of the estimated error to the true error (a measure of the quality of the estimate), grows with the simulation time.
In Section~\ref{sec:numerics}, we demonstrate through numerical examples that this growth typically  occurs at rate $t$ or $\sqrt{t}$ for a simulation of length $t$.
Since this growth is independent of the discretisation parameters it does not prevent the estimates being of asymptotically optimal order, but it is less than desirable from the perspective of obtaining good estimates of the discretisation error, particularly for long time simulations.
It is interesting to note that this growth does not occur for the duality-based $L^{\infty}(0,t; L^2(\domain))$ error estimates of~\cite{ERIKSSON:1991be,ERIKSSON:1995in}, and the cause of this problem can be directly attributed to the way in which the terms of the error estimate accumulate with time.
The true error accumulates through time in an $L^{\infty}(0,t)$ fashion, which is mimicked by the duality-based estimates where the various estimator components also accumulate through time in an $L^{\infty}(0,t)$ manner.
Estimators derived through energy arguments, however, typically involve terms measuring the various sources of discretisation error accumulating through time in $L^1(0,t)$ or $L^2(0,t)$ norms, leading to error estimators which grow rapidly even when the error itself may not.
The inherent problem with such an estimate can be seen by considering the situation where equally large errors are committed on each time step of the simulation: in this case, an $L^1(0,t)$ or $L^2(0,t)$-type accumulation of the error will grow with $t$, while an $L^{\infty}(0,t)$-type accumulation will remain constant.
It is this gap between rates of accumulation which manifests itself as ever-growing effectivities.

The main result of this paper, Theorem~\ref{thm:finalestimate}, addresses this shortcoming. 
New estimates are introduced for the error in the $L^{\infty}(0,t; L^2(\domain))$ norm, in which the individual estimator terms are permitted to accumulate through time in $L^p(0,t)$ norms for any $p \in [1,\infty]$.
The derivation of these estimates is based on energy techniques through elliptic reconstructions and, to the best of our knowledge, their structure is completely novel, providing a wide variety of different bounds simultaneously, almost all of which are new.
In particular, this family includes the familiar elliptic reconstruction-based estimates built on $L^1(0,t)$ and $L^2(0,t)$-type accumulations, whilst also allowing counterparts to the duality-based estimators incorporating $L^\infty(0,t)$-type accumulations to be derived using energy techniques.
The key advantage of this is that, since the estimator terms can therefore accumulate through time in the same norm as the error, they exhibit effectivities which in practice appear to remain bounded with respect to the simulation length, typically tending to some constant value.
As such, these estimates are very well suited for long time simulations.

The technique for deriving these estimates is fairly simple and essentially due to the structure of the partial differential equation.
Ultimately, it mirrors a classical technique, well known from the a priori error analysis of discretisations of parabolic problems, of using a Poincar\'e-Friedrichs inequality to remove the $H^1$ norm of the error from the left hand side of the error equation derived by energy techniques, providing instead a bounded exponential factor on each term of the estimate, cf.~\citet[Chapter 1]{Thomee:2006th} for the a priori viewpoint, or Lemma~\ref{lem:reconstructionError} for its manifestation here.
We note that an outwardly similar exponential factor appears in the final time error bounds in the $L^2(\domain)$ norm of~\citet[\S6]{Lakkis:2007du} and the final time $L^{\infty}(\domain)$ norm bounds of~\cite{Kopteva:2013hy,Kopteva:2017ib}.
The difference is in how this exponential factor is used; here, it allows us to apply a H\"older inequality over the time domain (cf. Lemma~\ref{lem:timeAccumulations}), producing an estimate with general $L^p(0,t)$-type accumulations, while in the aforementioned final time bounds it provides a `forgetfulness' property to the estimator which is broadly associated with the smoothing properties of the differential equation.

The technique is demonstrated here applied to numerical schemes comprised of backward Euler and Crank-Nicolson time-stepping schemes coupled with conforming finite elements in space, using fixed spatial and temporal discretisations.
We do not, however, foresee any extra significant difficulties in applying the technique to other classes of discretisations or associated adaptive schemes.
As an example, we refer the reader to~\cite{Sutton:2017vc} where the technique is applied in the context of an adaptive algorithm for very general polygonal meshes built on a spatial discretisation using a virtual element method~\citep{VEIGA:2013wi,Cangiani:2017fu,Sutton:2017dp} with backward Euler time-stepping.
The assumption here of a fixed spatial mesh, however, allows us to remove a layer of technicalities, and thus present a slightly simplified version of the quadratic time discretisation used to derive optimal order $L^{\infty}(0,t; L^2(\domain))$ estimates for Crank-Nicolson schemes by~\cite{Akrivis:2006hw,Lozinski:2009bt,Bansch:2012gu,Bansch:2013cu}.
We further refine the results of~\cite{Bansch:2012gu,Bansch:2013cu} by deriving estimates in which the data approximation term is of optimal order even when the forcing term is non-zero on the boundary of the domain.
This loss of optimality was due to the fact that the data approximation terms of the previous estimates involved the projection of the forcing data onto discrete functions which satisfy zero boundary values, which is suboptimal when the forcing function is non-zero on the boundary.

We begin by introducing our model parabolic problem in Section~\ref{sec:modelProblem}, and its discretisation in Section~\ref{sec:discretisation}.
The new class of a posteriori error estimates is then derived in Section~\ref{sec:aposteriori} and their practical performance is assessed through a set of numerical examples in Section~\ref{sec:numerics}.
Finally, we present some conclusions in Section~\ref{sec:conclusion}.

\section{Model problem}\label{sec:modelProblem}
Let $\finaltime > 0$, and let $\domain \subset \Re^{\spacedim}$ be a bounded open domain with $\spacedim = 2,3$.
For simplicity, we assume that the domain $\domain$ is a convex polygon, although the results we present could be extended to non-convex domains with reentrant corners through careful application of weighted estimates; see, for example,~\cite{Wihler:2007tv}.

The model problem we consider is to find $u : \domain \times [0,\finaltime] \to \Re$ satisfying
\begin{align*}
	\begin{split}
	u_t(x,t) - \diffop u(x,t) &= \force(x,t)\phantom{u_0(x)0} \text{ for } (x,t) \in \domain \times (0,\finaltime],
	\\
	u(x,0) &= u_0(x)\phantom{\force(x,t)0} \text{ for } x \in \domain,
	\\
	u(x,t) &= 0\phantom{u_0(x)\force(x,t)} \text{ for } (x,t) \in \boundary \times (0,\finaltime].
	\end{split}
\end{align*}
Here $\diffop$ is a symmetric positive definite linear second order elliptic operator of the form 
\begin{align*}
	\diffop v = \nabla \cdot (A \nabla v) - \mu v,
\end{align*}
where $\mu \in L^2(\domain \times [0,\finaltime])$ with $\mu(x,t) \geq 0$ for almost every $(x,t) \in \domain \times [0,\finaltime]$, and $A : \domain \times [0,\finaltime] \to \Re^{\spacedim \times \spacedim}$ is symmetric and positive definite for almost every $(x,t) \in \domain \times [0,\finaltime]$.
Let $\weakform : H^1_0(\domain) \times H^1_0(\domain) \to \Re$ denote the bilinear form
\begin{align*}
	\weakform(v,w) = (A \nabla v, \nabla w) + (\mu v,w),
\end{align*}
where $(v,w) \equiv \int_{\domain} vw \differential x$ denotes the $L^2(\domain)$ inner product.
Further, for $\omega \subset \Re^{m}$ and an integer $m > 0$, we use $\norm{\cdot}_{W^{k,p}(\omega)}$ and $\abs{\cdot}_{W^{k,p}(\omega)}$ to denote the standard norm and seminorm on the Sobolev space $W^{k,p}(\omega)$ for $k > 0$ and $p \in [1,\infty]$ (for further details see~\cite{Adams:2003wi}, for example).
For the special case of $\omega = \domain$, we shall denote the $L^2(\domain)$ norm by $\norm{\cdot}$ and the $H^k(\domain)$ norm by $\norm{\cdot}_{k}$.

We note that $\weakform$ induces a norm on $H^1_0(\domain)$, which we denote by $\weakform(v, v) = \triplenorm{v}^2$. 
We further observe that $\weakform$ is continuous in this norm, and under the assumptions above $\triplenorm{\cdot}$ is equivalent to the standard $H^1(\domain)$ norm, i.e. there exists a constant $\Cequiv \in \Re$ such that
\begin{align}\label{eq:normEquivalence}
	\Cequiv^{-1} \norm{\testfn}_1 \leq \triplenorm{\testfn} \leq \Cequiv \norm{\testfn}_1,
\end{align}
for all $\testfn \in H^1_0(\domain)$.

The problem can therefore be written in the weak form: find $u \in L^{2}(0,\finaltime; H^1_0(\domain))$ with $u_t \in L^{2}(0,\finaltime; H^{-1}(\domain))$ such that
\begin{align}
	(u_t(t), \testfn) + \weakform(u(t), \testfn) = (\force(t), \testfn) \text{ for all } \testfn \in H^1_0(\domain) \text{ and a.e. } t \in [0,\finaltime].
	\label{eq:continuousProblem}
\end{align}
Standard arguments ensure that this problem posesses a unique solution, cf.~\cite{Evans:2010ec}.

\section{Finite element discretisation}\label{sec:discretisation}

We now recall the ingredients of a conventional conforming finite element discretisation of~\eqref{eq:continuousProblem}.

\subsection{Discrete function space}
Suppose $\{t^{\currstep}\}_{\currstep = 0}^{\nsteps}$ forms a partition of $[0,\finaltime]$, with $t^{\currstep} = \currstep \timestep$ for $\currstep = 0,\dots,\nsteps$, where $\timestep = \frac{\finaltime}{\nsteps}$.
Let $\mesh$ be a shape-regular partition of the domain $\domain \subset \Re^{\spacedim}$ into non-overlapping elements which are either $\spacedim$-dimensional simplicies or hypercubes.
To avoid extra technicalities, we suppose that the elements in the mesh are either all simplicies or all hypercubes.
We denote by $\sides$ the \emph{skeleton} of the mesh $\mesh$; for $\domain \subset \Re^2$ this is the set of element edges, while for $\domain \subset \Re^3$ it is the set of element faces.
We further introduce the mesh dependent norms
\begin{align*}
	\norm{\cdot}_{\mesh} := \Big( \sum_{\el \in \mesh} \norm{ \cdot }_{L^2(\el)}^2 \Big)^{1/2} 
	\quad \text{ and } \quad
	\norm{\cdot}_{\sides} := \Big( \sum_{\side \in \sides} \norm{ \cdot }_{L^2(\side)}^2 \Big)^{1/2} ,
\end{align*}
and the piecewise constant mesh size function $h : \domain \to \Re$ such that $h|_{\el} = \operatorname{diam}(\el)$ for each $\el \in \mesh$ and $h|_{\side} = \operatorname{diam}(\side)$ for each $\side \in \sides$.

The conventional conforming finite element function spaces with respect to $\mesh$ are then given by
\begin{align*}
	\fespace = \{ v \in H^1(\domain) : v|_{\el} \in \fespace_{\el} \quad \forall \el \in \mesh \}
	\quad \text{ and } \quad
	\fespacebc = \fespace \cap H^1_0(\domain),
\end{align*}
where
\begin{equation*}
	\fespace_{\el} := 
	\begin{cases}
		\pbasis[\el]{\degree} &\text{if } \el \text{ is triangular}, 
		\\
		\qbasis[\el]{\degree} &\text{if } \el \text{ is quadrilateral}.
	\end{cases}
\end{equation*}
Here, $\pbasis[\el]{\degree}$ denotes the space of polynomials of total degree $\degree$ on $\el$, and $\qbasis[\el]{\degree}$ denotes the space of tensor-product polynomials of maximum degree $\degree$ on $\el$.

The assumptions on the mesh and discrete space above ensure the existence of a \emph{Cl\'ement-type interpolation operator}, satisfying the following approximation estimate.

\begin{lemma}[Cl\'ement-type interpolation estimate, cf.~\cite{Clement:1975te}]\label{lem:clementEstimates}
	For any $v \in H^1_0(\domain)$ there exists $v_I \in \fespacebc$ such that
	\begin{align*}
		\abs{v_I}_{H^1(\el)} &\leq \Cclem \abs{v}_{H^1(\widehat{\el})}
		\\
		\norm{h^{-1}(v - v_I)}_{\el} &\leq \Cclem \abs{v}_{H^1(\widehat{\el})}
		\\
		\norm{h^{-1/2}(v - v_I)}_{\side} &\leq \Cclem \abs{v}_{H^1(\widehat{\side})}
	\end{align*}
	for all $\el \in \mesh$ and $\side \in \sides$, where 
	$\Cclem$ is a positive constant depending only on the shape regularity of $\mesh$.
	Here, $\widehat{\el}$ and $\widehat{\side}$ denote the usual \emph{finite element patch} of $\el$ and $\side$, respectively, formed of all mesh elements with which they share a vertex.
\end{lemma}

Finally, for quantities which may be discontinuous across the mesh skeleton, we
define the jump operator $\jump{\cdot}$ across a mesh interface $\side\in\sides$ as follows. 
If $\side \cap \boundary = \emptyset$, then there exist $\el^{+}$ and $\el^{-}$ such that $\side \subset \partial\el^{+}\cap\partial\el^{-}$.
Denote by $\vec{v}^{\pm}$ the trace of the vector valued function $\vec{v}$ on $\side$ from within $\el^{\pm}$ and $\vec{n}_\side^{\pm}$ the unit outward normal on $\side$ with respect to $\el^{\pm}$.
Then, $\jump{\vec{v}}:=\vec{v}^+ \cdot \vec{n}_{\side}^{+} + \vec{v}^- \cdot \vec{n}_{\side}^{-}$.
On the other hand, if $\side \cap \boundary = \emptyset$ then $\jump{\vec{v}} = 0$.

\subsection{Discrete differential operators}
We define the \emph{discrete spatial operator} $\diffoph[\currstep] : \fespacebc[\currstep] \to \fespacebc[\currstep]$ to satisfy
\begin{align*}
	(\diffoph[\currstep] \wh, \testfnh[\currstep]) = - \weakform(\wh, \testfnh[\currstep]) \quad \forall \testfnh[\currstep] \in \fespacebc[\currstep],
\end{align*}
and, for each $n \in \{1,\dots,N\}$, the \emph{discrete time derivative} $\timederivh^{\currstep} : \fespace \to \fespace$, defined for a set of functions $\{\wh[\currstep]\}_{\currstep = 0}^{\nsteps} \subset \fespace[\currstep]$ by
\begin{align*}
	\timederivh^{\currstep} \wh[\currstep] := \frac{\wh[\currstep] - \wh[\prevstep]}{\timestep}.
\end{align*}
For the remainder of the article, we shall use the shorthand notation $\timederivh \equiv \timederivh^{\currstep}$ for brevity.
We also introduce the $L^2(\domain)$-orthogonal projectors $\ltwoproj : L^2(\domain) \to \fespace$ and $\ltwoprojbc : L^2(\domain) \to \fespacebc$, which satisfy
\begin{align*}
	(\ltwoproj v - v, \testfnh) = 0 \quad \forall \testfnh \in \fespace
	\qquad \text{ and } \qquad
	(\ltwoprojbc v - v, \testfnh) = 0 \quad \forall \testfnh \in \fespacebc,
\end{align*}
respectively.
The key difference between these is that the latter imposes zero Dirichlet boundary data.

\subsection{Discrete numerical schemes}
We now introduce the backward Euler and Crank-Nicolson schemes for computing discrete approximations to the solution of~\eqref{eq:continuousProblem}.
With clear ambiguity we shall refer to the solutions obtained at the time nodes $t^{\currstep}$ using either method as $\{\U[\currstep]\}_{\currstep = 0}^{\nsteps} \subset \fespacebc[\currstep]$, since throughout it shall be clear from the context which we are referring to.

\subsubsection{Backward Euler time discretisation}
The backward Euler method for approximating solutions to~\eqref{eq:continuousProblem} may be expressed as:
find the sequence of finite element functions $\{\U[\currstep]\}_{\currstep = 0}^{\nsteps} \subset \fespacebc[\currstep]$ satisfying
\begin{align*}
	\U[0] &= \interp u_0,
	\\
	(\timederivh \U[\currstep], \testfnh[\currstep]) + \weakform(\U[\currstep], \testfnh[\currstep]) &= (\force[\currstep], \testfnh[\currstep]) \quad \text{ for all } \testfnh[\currstep] \in \fespacebc[\currstep] \text{ and } \currstep = 1,\dots,\nsteps,
\end{align*}
where $\force[\currstep] := \force(t^{\currstep})$ and $\interp : L^2(\domain) \to \fespacebc[0]$ is a suitable interpolation or projection operator into the finite element space.
Using the definition of the discrete spatial operator, this numerical scheme may equivalently be written in the pointwise form: 
find $\{\U[\currstep]\}_{\currstep = 0}^{\nsteps} \subset \fespacebc[\currstep]$ with $\U[0] = \interp u_0$, and
\begin{align*}
	\timederivh \U[\currstep] - \diffoph \U[\currstep] &= \forcehbc[\currstep],
\end{align*}
for $\currstep = 1,\dots,\nsteps$.

\subsubsection{Crank-Nicolson time discretisation}
The Crank-Nicolson method for finding approximate solutions of~\eqref{eq:continuousProblem} may be expressed as:
find $\{\U[\currstep]\}_{\currstep = 0}^{\nsteps} \subset \fespacebc[\currstep]$ satisfying
\begin{align*}
	\U[0] &= \interp u_0,
	\\
	(\timederivh \U[\currstep], \testfnh[\currstep]) + \frac{1}{2}\weakform ( \U[\currstep] + \U[\prevstep], \testfnh[\currstep] ) &= ( \force[\midstep], \testfnh[\currstep] ) \quad \text{ for all } \testfnh[\currstep] \in \fespacebc[\currstep] \text{ and } \currstep = 1,\dots,\nsteps,
\end{align*}
where 
$\force[\midstep] = \force \big(\frac{t^{\currstep} + t^{\prevstep}}{2} \big)$ and $\interp : L^2(\domain) \to \fespacebc[0]$ is a suitable interpolation or projection operator into the finite element space.
As with the backward Euler method, we may write this scheme in the equivalent pointwise form:  
find $\{\U[\currstep]\}_{\currstep = 0}^{\nsteps} \subset \fespacebc[\currstep]$ with $\U[0] = \interp u_0$, and
\begin{align*}
	\timederivh \U[\currstep] - \frac{1}{2} \diffoph \U[\currstep] - \frac{1}{2} \diffoph \U[\prevstep] &= \forcehbc[\midstep],
\end{align*}
for $\currstep = 1,\dots,\nsteps$.

\section{A posteriori error estimation}\label{sec:aposteriori}
To derive the $L^{\infty}(0,t; L^2(\domain))$ error estimates for the backward Euler and Crank-Nicolson schemes of Theorem~\ref{thm:finalestimate}, we proceed in several stages.
We begin in Section~\ref{sec:reconstructions} by introducing the space, time, and space-time reconstructions of the discrete solutions which will be required to derive the error estimates. 
These are used in Section~\ref{sec:errorEquations} to derive differential equations satisfied by the error which have controllable right hand sides and thus form the basis of our error estimates.
Some tools for producing error estimates composed of very flexible varieties of time accumulation are discussed in Section~\ref{sec:timeAccumulations}, and in Section~\ref{sec:parabolicError} these are applied to the error equations of Section~\ref{sec:errorEquations} to derive estimates on the parabolic component of the error.
Finally, in Section~\ref{sec:finalEstimates} we put the pieces together to derive estimates on the total error.

\subsection{Space-time reconstruction operators}\label{sec:reconstructions}

The forthcoming analysis rests on the use of various reconstruction operators in space and time, which we define here.
We begin by defining the spatial reconstruction we use, which is the celebrated elliptic reconstruction operator of~\cite{Makridakis:2003ws}.

\begin{definition}[Elliptic reconstruction]
Let $\eliprecon[], \eliprecon[\currstep] : \fespacebc \to H^1_0(\domain)$ denote the \emph{elliptic reconstruction operators}, respectively satisfying
\begin{align*}
	\weakform(\eliprecon[] \alttestfn_h, \testfn) = (- \diffoph \alttestfn_h, \testfn)
	\quad \text{ and } \quad
	\weakform(\eliprecon[\currstep] \alttestfn_h, \testfn) = (- \diffoph \alttestfn_h + \forceh[\currstep] - \forcehbc[\currstep], \testfn)
\end{align*}
for all $\testfn \in H^1_0(\domain)$.
\end{definition}

We note that the definitions of $\eliprecon[]$ and $\eliprecon[\currstep]$ coincide when $\forceh[\currstep]|_{\boundary} = 0$, since in this case $\forceh[\currstep] - \forcehbc[\currstep] = 0$.
When $\forceh[\currstep]|_{\boundary} \neq 0$, however, the second definition is required to ensure that the data oscillation term in the final error estimate is of optimal order.

The key property of the elliptic reconstruction operator is that any $\alttestfn_h \in \fespacebc$ can now be seen as satisfying the finite element discretisation of the elliptic problem satisfied by $\eliprecon[] \alttestfn_h$.
This allows us to utilise the well developed literature on a posteriori error estimates for elliptic problems in order to derive bounds on quantities of the form $\norm{\alttestfn_h - \eliprecon[] \alttestfn_h}$.
To demonstrate this, in Lemma~\ref{lem:ellipreconBound} we provide an example of a residual-type a posteriori error estimate, which may be proven using standard techniques (see, for example,~\citet{Ainsworth:2000cv,Verfurth:uz}, etc.).
It is worth stressing that the choice of a residual-type bound here is by no means the only possibility --- a great strength of the elliptic reconstruction technique is that any methodology for designing a posteriori estimates for elliptic problems can be applied here.

\begin{lemma}[Elliptic reconstruction error estimate]\label{lem:ellipreconBound}
	There exists a constant $\Celip$, depending only on the domain $\domain$, the problem data, and the regularity of the mesh $\mesh$, such that for any $\alttestfnh \in \fespacebc$, the elliptic reconstruction operators $\eliprecon[]$ and $\eliprecon[\currstep]$ satisfy
	\begin{align*}
		\norm{\alttestfnh - \eliprecon[] \alttestfnh} 
		&\leq 
		\ellipest(\alttestfnh) 
		:= 
		\Celip \big(\norm{ h^2 (\diffoph \alttestfnh - \diffop \alttestfnh)}_{\mesh}  
		+ \norm{ h^{3/2} \jump{ A \nabla \alttestfn_h } }_{\sides} \big),
	\end{align*}
	and
	\begin{align*}
		\norm{\alttestfnh - \eliprecon[\currstep] \alttestfnh} 
		&\leq 
		\ellipest(\alttestfnh) 
		:= 
		\Celip \big(\norm{ h^2 (\diffoph \alttestfnh + \forcehbc[\currstep] - \forceh[\currstep] - \diffop \alttestfnh)}_{\mesh}  
		+ \norm{ h^{3/2} \jump{ A \nabla \alttestfn_h } }_{\sides} \big),
	\end{align*}
	respectively. 
\end{lemma}

To introduce the time reconstructions which will also be required for the final error estimates, we first introduce the so-called \emph{temporal hat functions}, continuous linear functions $\lcurr : [0,T] \to [0,1]$ for each $\currstep = 0,\dots,N$ satisfying $\ell^{i}(t^{j}) = \delta_{ij}$, where $\delta_{ij}$ is Kronecker's delta.
These are defined as
\begin{align*}
	\lcurr(t) = 
	\begin{cases}
		\frac{t - t^{\prevstep}}{\timestep} & \text{ for } t \in [t^{\prevstep}, t^{\currstep}],
		\\
		\frac{t^{\currstep+1} - t}{\timestep} & \text{ for } t \in [t^{\currstep}, t^{\currstep+1}],
		\\
		\quad 0 & \text{ otherwise},
	\end{cases}
\end{align*}
and will be used to patch together the nodal discrete solutions $\U[\currstep]$ provided by the discrete schemes, through the piecewise-polynomial time and space-time reconstructions given in Definition~\ref{def:timeReconstructions}.
The quadratic reconstruction we adopt here is a slight simplification of the `two-point reconstruction' studied by~\cite{Akrivis:2006hw,Lozinski:2009bt,Bansch:2012gu,Bansch:2013cu}.

\begin{definition}[Time and space-time reconstructions]\label{def:timeReconstructions}
	Let $\{\U[\currstep]\}_{\currstep = 0}^{\nsteps} \subset \fespacebc[\currstep]$ denote the set of discrete approximate solutions produced by either the backward Euler or Crank-Nicolson scheme.
	We let $\lintimerecon, \quadtimerecon : [0,T] \to \fespacebc$ denote the \emph{linear} and \emph{quadratic time reconstructions}, given by
	\begin{align*}
		\lintimerecon(t) = \lcurr(t) \U[\currstep] + \lprev(t) \U[\prevstep]
		\quad\text{ and }\quad
		\quadtimerecon(t) = \lintimerecon(t) - \frac{\timestep^2}{2} \lcurr(t) \lprev(t) \timederivh (\diffoph[\currstep] \U[\currstep] + \forcehbc[\currstep]),
	\end{align*}
	respectively, where $t \in [t^{\prevstep}, t^{\currstep}]$ for each $\currstep \in \{1,\dots,N\}$.
	We further let $\linspacetimerecon, \quadspacetimerecon : [0,T] \to \fespacebc$ denote the \emph{linear} and \emph{quadratic space-time reconstructions}, given by
	\begin{align*}
		\linspacetimerecon(t) = \lcurr(t) \eliprecon[\currstep] \U[\currstep] + \lprev(t) \eliprecon[\prevstep] \U[\prevstep]
		\quad\text{ and }\quad
		\quadspacetimerecon(t) = \linspacetimerecon(t) - \frac{\timestep^2}{2} \lcurr(t) \lprev(t) \eliprecon[] \timederivh (\diffoph[\currstep] \U[\currstep] + \forcehbc[\currstep]),
	\end{align*}
	respectively, where $t \in [t^{\prevstep}, t^{\currstep}]$ for each $\currstep \in \{1,\dots,N\}$.
\end{definition}
We remark that all of these reconstructions are continuous in time. 
Indeed, for each $\currstep \in \{0,\dots,N\}$ the time reconstructions satisfy 
$\lintimerecon(t^{\currstep}) = \quadtimerecon(t^{\currstep}) = \U[\currstep]$, while the space-time reconstructions satisfy
 $\linspacetimerecon(t^{\currstep}) = \quadspacetimerecon(t^{\currstep}) = \eliprecon[\currstep] \U[\currstep]$.

\subsection{Error equations}\label{sec:errorEquations}
The final error estimate of Theorem~\ref{thm:finalestimate} will be proven by splitting the error into components using the triangle inequality with the reconstructions of Definition~\ref{def:timeReconstructions}, and controlling each term separately.
For the backward Euler method, this splitting is
\begin{align}\label{eq:be:splitting}
	\norm{u - \lintimerecon}_{L^{\infty}(0,t;L^2(\domain))} \leq \norm{u - \linspacetimerecon}_{L^{\infty}(0,t;L^2(\domain))} + \norm{\linspacetimerecon - \lintimerecon}_{L^{\infty}(0,t;L^2(\domain))},
\end{align}
while for the Crank-Nicolson method the splitting is
\begin{align}\label{eq:cn:splitting}
	\norm{u - \lintimerecon}_{L^{\infty}(0,t;L^2(\domain))} \leq \norm{u - \quadspacetimerecon}_{L^{\infty}(0,t;L^2(\domain))} + \norm{\quadspacetimerecon - \quadtimerecon}_{L^{\infty}(0,t;L^2(\domain))} + \norm{\quadtimerecon - \lintimerecon}_{L^{\infty}(0,t;L^2(\domain))}.
\end{align}
The first component in each case is known as the \emph{parabolic error}, and measures the error between the true solution $u$ and the reconstruction.
The focus of this section is to derive an appropriate \emph{error equation}, a differential equation satisfied by the parabolic error, which will be used to bound the first term of each splitting.
This is performed separately in Section~\ref{sec:be:errorEquation} for the backward Euler scheme and in Section~\ref{sec:cn:errorEquation} for the Crank-Nicolson scheme.
The second term of each splitting, known as the \emph{space reconstruction error}, measures the difference between a discrete solution and its reconstruction, and may be bounded using Lemma~\ref{lem:ellipreconBound}.
Finally, the third term of the Crank-Nicolson splitting measure the \emph{time reconstruction error}, which is bounded by directly using the definition of the quadratic time reconstruction.

\subsubsection{Backward Euler method}
\label{sec:be:errorEquation}

Invoking the definition of the discrete Laplacian and linear time reconstruction, we observe that the discrete solutions produced by the backward Euler scheme satisfy
\begin{align*}
	\lintimerecon_t(t) - \diffoph \lintimerecon[\currstep] = \forcehbc[\currstep],
\end{align*}
or, using the definition of the elliptic reconstruction,
\begin{align}\label{eq:be:reconstructedForm}
	(\lintimerecon_t(t), v) + \weakform(\eliprecon[\currstep] \lintimerecon[\currstep], v) = (\forceh[\currstep], v) \qquad \forall v \in H^1_0(\domain).
\end{align}
We observe that the effect of including projections of the forcing data in the definition of the reconstruction is that right hand side of this differential equation no longer involves the projection of $\force$ onto functions with zero boundary data, a fact which is crucial for obtaining error estimates with optimal order data approximation terms.

Subtracting~\eqref{eq:be:reconstructedForm} from the weak form~\eqref{eq:continuousProblem} of the PDE, we deduce that the \emph{backward Euler parabolic error} $\xi(t) := u(t) - \linspacetimerecon(t)$ satisfies the error equation
\begin{align}
	(\xi_t, \testfn) + \weakform(\xi, \testfn) 
	&= 
	(\lintimerecon_t - \eliprecon[] \lintimerecon_t, v) + \weakform(\eliprecon \U[\currstep] - \linspacetimerecon(t), \testfn) + (\force - \forceh[\currstep], \testfn)
	\label{eq:be:errorequation}
\end{align}
for any $\testfn \in H^1_0(\domain)$.

The terms on the right hand side of~\eqref{eq:be:errorequation} may be controlled by the (computable) error estimator functionals given in Definition~\ref{def:be:estimatorterms}, as shown in Lemma~\ref{lem:be:individualterms}.
To keep the notation from becoming too obscure, we denote each component of the estimator by the first letter of its name in a calligraphic font, with the subscript `BE' to indicate that it relates to the backward Euler method.

\begin{definition}[Terms of the backward Euler error estimate]\label{def:be:estimatorterms}
	Let $\currstep \in \{1,\dots,N\}$ and $t \in [t^{\prevstep}, t^{\currstep}]$.
	We define the \emph{space error estimator} and \emph{elliptic reconstruction error estimator} as
	\begin{align*}
		\spaceest{BE}(t) = \ellipest(\lintimerecon_t(t)),
		\quad\text{ and }\quad
		\ellipreconest{BE}(t) = \ellipest(\lintimerecon(t)),
	\end{align*}
	respectively, with $\ellipest$ denoting the elliptic error estimator from Lemma~\ref{lem:ellipreconBound}, the \emph{time error estimator} as
	\begin{align*}
		\timeest{BE}(t) = \timestep \norm{\timederivh (\diffoph[\currstep]\U[\currstep] - \forceh[\currstep] + \forcehbc[\currstep])},
	\end{align*}
	and the \emph{data approximation error estimators} for \emph{time} and \emph{space} as
	\begin{align*}
		\dataest{T, BE}(t) = \norm{\force(t) - \force[\currstep]},
		\qquad\text{ and }\qquad
		\dataest{S, BE}(t) = \Cclem \norm{h (\force[\currstep] - \forceh[\currstep])},
	\end{align*}
	respectively.
\end{definition}

\begin{lemma}[Estimates for individual backward Euler error equation terms]\label{lem:be:individualterms}
	Let $n \in \{1,\dots,N\}$ and $t \in [t^{\prevstep}, t^{\currstep}]$. 
	Then, the terms of the Crank-Nicolson error equation~\eqref{eq:be:errorequation} may be bounded with separate contributions from the spatial and temporal errors
	\begin{align*}
		((\lintimerecon - \linspacetimerecon)_{t}(t), v) \leq \spaceest{BE}(t) \norm{v}, 
		\quad \text{ and } \quad
		\weakform(\eliprecon[\currstep] \U[\currstep] - \linspacetimerecon(t), v) \leq \timeest{BE}(t) \norm{v},
	\end{align*}
	respectively,
	and the data approximation error
	\begin{align*}
		(\force(t) - \forceh[\currstep], \ctserr) &\leq \dataest{T, BE}(t) \norm{v} + \dataest{S, BE}(t) \triplenorm{v}.
	\end{align*}
\end{lemma}
\begin{proof}
    The first term on the right hand side of~\eqref{eq:be:errorequation} is bounded by applying the Cauchy-Schwarz inequality and Lemma~\ref{lem:ellipreconBound}.
	Since $\lprev(t) = 1 - \lcurr(t)$ for $t \in [t^{\prevstep}, t^{\currstep}]$, the second term of~\eqref{eq:be:errorequation} may be rewritten as
	\begin{align*}
		\weakform(\eliprecon[\currstep] \U[\currstep] - \linspacetimerecon, v) &= \lprev(t) \weakform(\eliprecon[\currstep] \U[\currstep] - \eliprecon[\prevstep] \U[\prevstep], v)
			= - \timestep \lprev(t) (\timederivh( \diffoph \U[\currstep] + \forcehbc[\currstep] - \forceh[\currstep]), v),
	\end{align*}
	and the result follows by applying the Cauchy-Schwarz inequality and using the fact that $\lprev(t) \leq 1$.
	
	The final term of~\eqref{eq:be:errorequation} may be bounded by adding and subtracting $\force[\currstep]$, recalling that $\forceh[\currstep] = \ltwoproj \force[\currstep]$, and exploiting the $L^2(\domain)$ orthogonality of the projector $\ltwoproj$ to introduce the Cl\'ement interpolant $v_h \in \fespacebc$ of $v$, providing
	\begin{align*}
		(\force - \forceh[\currstep], v) 
		= (\force - \force[\currstep], v) + (\force[\currstep] - \forceh[\currstep], v - v_h).
	\end{align*}
	The result then follows by applying Cauchy-Schwarz inequality and the bounds of Lemma~\ref{lem:clementEstimates}.
\end{proof}

\subsubsection{Crank-Nicolson method}
\label{sec:cn:errorEquation}
The definition of the quadratic time reconstruction operator $\quadtimerecon$ ensures that it satisfies the differential equation
\begin{align*}
	\quadtimerecon_t(t) - \diffoph \lintimerecon(t) = \forcehbc(t) + (\forcehbc[\midstep] - \forcehbc(t^{\midstep})),
\end{align*}
where $\forcehbc(t) = \lcurr(t) \forcehbc[\currstep] + \lprev(t) \forcehbc[\prevstep]$, and the elliptic reconstruction operator therefore implies that
\begin{align}\label{eq:Qpde}
	(\quadtimerecon_t(t), v) + \weakform(\linspacetimerecon(t), v) = (\force_h(t), v) + (\forcehbc[\midstep] - \forcehbc(t^{\midstep}), v) \qquad \forall v \in H^1_0(\domain),
\end{align}
with $\forceh(t) = \lcurr(t) \forceh[\currstep] + \lprev(t) \forceh[\prevstep]$.

Taken together, the weak form~\eqref{eq:continuousProblem} of the PDE and relation~\eqref{eq:Qpde} imply that the \emph{Crank-Nicolson parabolic error} $\ctserr(t) := u(t) - \quadspacetimerecon(t)$ satisfies the error equation
\begin{align}
	(\ctserr_t, \testfn) + \weakform(\ctserr, \testfn) 
	&= 
	((\quadtimerecon - \quadspacetimerecon)_{t}, \testfn) + \weakform(\linspacetimerecon - \quadspacetimerecon, \testfn) + (\force - \force_h, \testfn) + (\forcehbc[\midstep] - \forcehbc(t^{\midstep}), \testfn)
	\label{eq:cn:errorequation}
\end{align}
for any $\testfn \in H^1_0(\domain)$.

The terms on the right hand side of~\eqref{eq:cn:errorequation} may be controlled by the error estimator functionals defined in Definition~\ref{def:cn:estimatorterms}, as shown in Lemma~\ref{lem:cn:individualterms}.
Once again, we denote each of the estimators by the first letter of its name in a calligraphic font, this time with the subscript `CN' to indicate that it relates to the Crank-Nicolson method.

\begin{definition}[Terms of the Crank-Nicolson error estimate]\label{def:cn:estimatorterms}
	Let $\currstep \in \{1,\dots,N\}$ and $t \in [t^{\prevstep}, t^{\currstep}]$.
	We define the \emph{space error estimator} and \emph{elliptic reconstruction error estimator} as
	\begin{align*}
		\spaceest{CN}(t) = \ellipest(\quadtimerecon_t(t)),
		\quad\text{ and }\quad
		\ellipreconest{CN}(t) = \ellipest(\quadtimerecon(t)),
	\end{align*}
	respectively, with $\ellipest$ the elliptic error estimator from Lemma~\ref{lem:ellipreconBound}, 
	the \emph{time error estimator} as
	\begin{align*}
		\timeest{CN}(t) = \Cclem \frac{\timestep^2}{8} \Big( \triplenorm{\timederivh (\diffoph[\currstep]\U[\currstep] + \forcehbc[\currstep])}
			+ 
			\norm{h \diffoph[\currstep] \timederivh (\diffoph[\currstep]\U[\currstep] + \forcehbc[\currstep])} \Big),
	\end{align*}
	the \emph{quadratic time reconstruction estimator} as
	\begin{align*}
		\timereconest{CN}(t) = \frac{\timestep^2}{8} \norm{\timederivh (\diffoph[\currstep] \U[\currstep] + \forcehbc[\currstep])},
	\end{align*}
	and the \emph{data approximation error estimators} for \emph{time} and \emph{space} as
	\begin{align*}
		\dataest{T, CN}(t) &= \norm{\force(t) - \lcurr(t) \force[\currstep] - \lprev(t) \force[\prevstep]} + \norm{\forcehbc[\midstep] - \forcehbc(t^{\midstep})},
		\\
		\dataest{S, CN}(t) &= \Cclem \norm{h (\lcurr(t) (\force[\prevstep] - \forceh[\prevstep]) + \lprev(t) (\force[\currstep] - \forceh[\currstep]))},
	\end{align*}
	respectively.
\end{definition}

\begin{lemma}[Estimates for individual Crank-Nicolson error equation terms]\label{lem:cn:individualterms}
	The terms of the Crank-Nicolson error equation~\eqref{eq:cn:errorequation} may be bounded with separate contributions from the spatial and temporal errors
	\begin{align*}
		((\quadtimerecon - \quadspacetimerecon)_{t}(t), v) \leq \spaceest{CN}(t) \norm{v},
		\quad \text{ and } \quad
		\weakform(\linspacetimerecon(t) - \quadspacetimerecon(t), \ctserr) \leq \timeest{CN}(t) \triplenorm{v},
	\end{align*}
	respectively,
	and the data approximation error
	\begin{align*}
		(\force(t) - \forceh(t), \ctserr) + (\forcehbc[\midstep] - \forcehbc(t^{\midstep}), \testfn) &\leq \dataest{T, CN}(t) \norm{v} + \dataest{S, CN}(t) \triplenorm{v}.
	\end{align*}
\end{lemma}
\begin{proof}
	The proof for the space and data approximation errors are broadly similar to the proofs given in Lemma~\ref{lem:be:individualterms}, so are omitted.
	For the bound on the time error, we first observe that
	\begin{align*}
		\weakform(\linspacetimerecon - \quadspacetimerecon, v) = \frac{\timestep^2}{2} \lcurr(t) \lprev(t) \weakform(\eliprecon[] \timederivh( \diffoph \U[\currstep] + \forcehbc[\currstep]), v).
	\end{align*}
	Letting $v_h \in \fespacebc$ denote the Cl\'ement interpolant of $v$, we obtain
	\begin{align*}
		\weakform(\eliprecon[] \timederivh( \diffoph \U[\currstep] + \forcehbc[\currstep]), v) 
		&= 
		\weakform(\eliprecon[] \timederivh( \diffoph \U[\currstep] + \forcehbc[\currstep]), v - v_h) + \weakform(\eliprecon[] \timederivh( \diffoph \U[\currstep] + \forcehbc[\currstep]), v_h)
		\\
		&=
		- (\diffoph \timederivh( \diffoph \U[\currstep] + \forcehbc[\currstep]), v - v_h) + \weakform(\timederivh( \diffoph \U[\currstep] + \forcehbc[\currstep]), v_h).
	\end{align*}
	The result then follows from continuity of $\weakform$, the approximation and stability properties of the Cl\'ement interpolant laid out in Lemma~\ref{lem:clementEstimates},
	and the fact that $\lcurr(t)\lprev(t) \leq \frac{1}{4}$.
\end{proof}

\subsection{Time accumulations}\label{sec:timeAccumulations}
We now introduce some machinery which will be required to derive a wide variety of new error estimates.
We first introduce the \emph{accumulation control coefficients}, increasing bounded functions of time which control the rate at which error accumulates and have a significant effect on the estimate at early times, as demonstrated in Section~\ref{sec:numerics}.

\begin{definition}[Accumulation control coefficients]
	\label{def:timeAccumulations}
	Let $p \in [1,\infty]$, $\lambda \in [0,1]$, and let $r \in [0,T]$.
	We introduce the 
	\emph{accumulation control coefficients} 
	\begin{align*}
		c_{p,r} := \norm{\beta_r}_{L^q(0,r)}
		=
		\begin{cases}
			 \Big( \dfrac{1 - e^{-q \alpha_\lambda r}}{q \alpha_\lambda} \Big)^{1/q} &\text{ for } p \in (1,\infty],
			 \\
			 \qquad 1 &\text{ for } p = 1,
		\end{cases}
	\end{align*}
	where $q$ satisfies $\frac{1}{p} + \frac{1}{q} = 1$, $\beta_r(s) = e^{\alpha_\lambda(s-r)}$, and $\alpha_\lambda = \frac{2(1 - \lambda) }{(\Cequiv\Cpf)^2}$.
\end{definition}

The rationale for introducing these accumulation control coefficients is demonstrated by Lemma~\ref{lem:timeAccumulations}.
Here, one should think of the term $F$ as representing some estimator term accumulating with the simulation time, and $\errorSurrogate$ as the error which we are trying to bound (either in the $L^{\infty}(L^2)$ or a weighted $L^2(H^1)$ norm).
The terms being bounded in each case are typical terms which we shall encounter in the subsequent error analysis, cf. Lemma~\ref{lem:reconstructionError}.

\begin{lemma}[Exponentially weighted time accumulations]
	\label{lem:timeAccumulations}
	Let $r \geq 0$ and suppose that $F \in L^{p^{\star}}(0,r)$ for some $p^{\star} \in [1,\infty]$, with $F(t) \geq 0$ for a.e. $t \in [0,r]$.
	Then, for $\errorSurrogate \in L^{\infty}(0,r; L^2(\domain))$, the estimate
	\begin{align}\label{eq:lplinfacc}
		\int_{0}^{r} 
		&e^{\alpha_\lambda (s - r)} 
		F(s)
		\norm{\errorSurrogate(s)}
		\differential s
		\leq
		\min_{p \in [1,\infty]} c_{p,r} \norm{F}_{L^p(0,r)}
		\max_{s \in [0, r]} \norm{\errorSurrogate(s)},
	\end{align}
	holds, and for $\errorSurrogate \in L^{2}(0,r; H^1(\domain))$, the estimate
	\begin{align}\label{eq:lpl2acc}
		\int_{0}^{r} 
		&e^{\alpha_\lambda (s - r)} 
		F(s)
		\triplenorm{\errorSurrogate(s)}
		\differential s
		\leq
		\min_{p \in [2,\infty]} (c_{\frac{p}{2},r})^{1/2} \norm{F}_{L^{p}(0,r)}
		\Big( \int_0^r e^{\alpha_\lambda (s - r)} \triplenorm{\errorSurrogate(s)}^2 \differential s \Big)^{1/2},
	\end{align}
	holds.
	The accumulation control coefficients $c_{p,r}$ are defined in Definition~\ref{def:timeAccumulations}.
\end{lemma}
\begin{proof}
	We begin by proving~\eqref{eq:lplinfacc}.
	Taking the maximum of $\norm{\errorSurrogate(t)}$ and applying H\"older's inequality for some $p \in [1, p^{\star}]$, we find
	\begin{align*}
		\int_{0}^{r} 
		&e^{\alpha_\lambda (s - r)} 
		F(s)
		\norm{\errorSurrogate(s)}
		\differential s
		\leq
		\max_{s \in [0, t^m]} \norm{\errorSurrogate(s)}
		\int_{0}^{r} 
		e^{\alpha_\lambda (s - r)} 
		F(s)
		\differential s
		\leq
		\norm{\beta_r}_{L^q(0,r)}
		\norm{F}_{L^p(0,r)}
		\max_{s \in [0, t^m]} \norm{\errorSurrogate(s)}.
	\end{align*}
	where $\beta_r(t) = e^{\alpha_\lambda (t - r)}$ and $q$ satisfies $\frac{1}{p} + \frac{1}{q} = 1$. 
	The result follows from the fact that $p$ is arbitrary, and the observation that $c_{p,r} = \norm{\beta_r}_{L^q(0,r)}$.
	
	To prove~\eqref{eq:lpl2acc}, we argue similarly. 
	Applying the Cauchy-Schwarz inequality, we find
	\begin{align*}
		\int_{0}^{r} 
		&e^{\alpha_\lambda (s - r)} 
		F(s)
		\triplenorm{\errorSurrogate(s)}
		\differential s
		\leq
		\Big( \int_{0}^{r} 
		e^{\alpha_\lambda (s - r)} 
		\triplenorm{\errorSurrogate(s)}^2
		\differential s \Big)^{1/2}
		\Big( \int_{0}^{r} 
		e^{\alpha_\lambda (s - r)} 
		(F(s))^2
		\differential s \Big)^{1/2}.
	\end{align*}
	H\"older's inequality, applied to the second term for some $p \in [1, p^{\star}]$, then implies that
	\begin{align*}
		\Big( \int_{0}^{r} 
		e^{\alpha_\lambda (s - r)} 
		(F(s))^2
		\differential s \Big)^{1/2}
		\leq
		\norm{\beta_r}_{L^q(0,r)}^{1/2}
		\norm{F}_{L^{2p}(0,r)},
	\end{align*}
	where $q$ once again satisfies $\frac{1}{p} + \frac{1}{q} = 1$.
	The result then follows from the definition of $c_{p,r}$ and the fact that $p$ is arbitrary.
\end{proof}

A few comments on this result are in order at this point.
Firstly, we note that if $F$ is constant on each interval, say $F(t) = F^{\currstep} \geq 0$ on $t \in (t^{\prevstep}, t^{\currstep})$, then the accumulation becomes
\begin{align*}
	\norm{F}_{L^p(0,t^m)} =
	\begin{cases}
	\left( \sum_{\currstep = 1}^{m} \timestep \big( F^{\currstep} \big)^{p} \right)^{1/p} &\text{ for } p \in [1,\infty)
	\\
	\max_{n \in \{1,\dots,m\} } F^{\currstep} &\text{ for } p = \infty,
	\end{cases}
\end{align*}
which, for $p=1,2$, yields the time accumulations familiar from conventional $L^{\infty}(0,T; L^2(\domain))$ error estimates such as those in~\cite{Lakkis:2006jk,Lozinski:2009bt,Bansch:2012gu,Bansch:2013cu,Karakatsani:2012ci,Georgoulis:2017wq}, for instance.

We next observe that while both bounds~\eqref{eq:lplinfacc} and~\eqref{eq:lpl2acc} allow for $L^{\infty}$-type time accumulations of $F$, only the former supports $L^1$-type accumulations.
As we shall demonstrate in Section~\ref{sec:numerics}, $L^1$-type accumulations are generally preferable for short time estimates.
In most situations, one has a choice which of the terms bounded in Lemma~\ref{lem:timeAccumulations} will appear in the error equation, the latter often coming with an extra power of $h$ (for instance, observe the bound of the data approximation term in Lemma~\ref{lem:reconstructionError}).
In this light, the choice may be seen as an alternative perspective on the conventional tradeoff between powers of $h$ or $\timestep$.
Since the aim of this work is to present \emph{long} time estimates, we shall not trouble ourselves too much on this point, although it is worth bearing in mind.

Finally, we remark that it is possible to subdivide the domain $[0,r]$ used in Lemma~\ref{lem:timeAccumulations}, and apply different types of accumulations on each sub division.
For instance, if $r_0 = 0$, $ r_M = r$, and $r_{i-1} < r_i$ for $i = 1,\dots,M$, we could obtain a bound of the form
\begin{align*}
		\int_{0}^{r} 
		&e^{\alpha_\lambda (s - r)} 
		F(s)
		\norm{\errorSurrogate(s)}
		\differential s
		\leq
		\max_{s \in [0, r]} \norm{\errorSurrogate(s)}
		\sum_{i=1}^{M}
		\Big(
		\min_{p_i \in [1,\infty]} c_{p_i,[r_{i-1},r_i]} \norm{F}_{L^{p_i}(r_{i-1},r_i)}
		\Big),
\end{align*}
where, here only, we use the notation $c_{p,[a, b]} = \norm{\beta_r}_{L^q(a,b)}$ for $a,b \in [0,r]$, with $\beta_r$ and $q$ defined as in Definition~\ref{def:timeAccumulations}.
It is clear that this bound is tighter than the bound~\eqref{eq:lplinfacc}, and could allow very general error estimates.
In the interests of clarity, however, this is not something we pursue any further here.

\subsection{Estimates for the parabolic error}\label{sec:parabolicError}

We now draw together the tools of Lemmas~\ref{lem:be:individualterms}, \ref{lem:cn:individualterms} and~\ref{lem:timeAccumulations} to provide an estimate for the parabolic error in each method.
While the previous lemmas provided spatial estimates for individual terms, in some sense, this lemma provides a temporal estimate and is the key point of departure in the proof from conventional $L^{\infty}(L^2)$ error estimates as it is here that the role of the new varieties of time accumulation first appear.
The first step of the argument, in which a Poincar\'e-Friedrichs inequality is applied, is a well known technique for deriving a priori error estimates for this class of problem, cf.~\cite{Thomee:2006th}, and appears to have been first applied in an a posteriori setting in a slightly different way to ultimately derive the final time $L^2(\domain)$ error estimate of~\citet[\S 6]{Lakkis:2007du}.

\begin{lemma}[$L^{\infty}(0,T; L^2(\domain))$ parabolic error estimates]\label{lem:reconstructionError}
    For any $r \in [0,T]$, the backward Euler parabolic error $\xi$ and Crank-Nicolson parabolic error $\ctserr$ satisfy
	\begin{align*}
		\max_{s \in [0,r]} \norm{\xi(s)}
		&\leq 
			\norm{\xi(0)} + 
			\sqrt{2} 
			\Big(
			\min_{p \in [1,\infty]} 
			c_{p,r}
			\norm{\spaceest{BE}}_{L^{p}(0,r)}
			+ 
			\min_{p \in [1,\infty]} 
			c_{p,r}
			\norm{\timeest{BE}}_{L^p(0,r)}
		\\&\qquad\qquad\qquad\qquad\qquad
			+ 
			\min_{p \in [1,\infty]} 
			c_{p,r}
			\norm{\dataest{T, BE}}_{L^p(0,r)}
			+ 
			\min_{p \in [2,\infty]} 
			(c_{\frac{p}{2},r})^{1/2}
			\norm{\dataest{S, BE}}_{L^p(0,r)}
			\Big),
	\end{align*}
	and
	\begin{align*}
		\max_{s \in [0,r]} \norm{\ctserr(s)}
		&\leq 
			\norm{\ctserr(0)} + 
			\sqrt{2} 
			\Big(
			\min_{p \in [1,\infty]} 
			c_{p,r}
			\norm{\spaceest{CN}}_{L^{p}(0,r)}
			+ 
			\min_{p \in [2,\infty]} 
			(c_{\frac{p}{2},r})^{1/2}
			\norm{\timeest{CN}}_{L^p(0,r)}
		\\&\qquad\qquad\qquad\qquad\qquad
			+ 
			\min_{p \in [1,\infty]} 
			c_{p,r}
			\norm{\dataest{T, CN}}_{L^p(0,r)}
			+ 
			\min_{p \in [2,\infty]} 
			(c_{\frac{p}{2},r})^{1/2}
			\norm{\dataest{S, CN}}_{L^p(0,r)}
			\Big),
	\end{align*}
	respectively.
\end{lemma}
\begin{proof}
    We prove the Crank-Nicolson estimate, and the backward Euler estimate follows similarly.

	To begin, we select $v = \ctserr$ in~\eqref{eq:cn:errorequation}, apply the estimates of Lemma~\ref{lem:cn:individualterms} to the terms on the right hand side individually.
	We therefore find that for any $\currstep \in \{ 1,\dots,N \}$ and $t \in [t^{\prevstep}, t^{\currstep}]$, 
	\begin{align*}
		\frac{1}{2}\frac{d}{dt} \norm{\ctserr(t)}^2 + \triplenorm{\ctserr(t)}^2 
		&\leq 
			\Big(
				\spaceest{CN}(t) 
				+
				\dataest{T, CN}(t) 
			\Big) \norm{\ctserr(t)}
			+
			\Big(
				\timeest{CN}(t) 
				+
				\dataest{S, CN}(t) 
			\Big) \triplenorm{\ctserr(t)}.
	\end{align*}
	The Poincar\'e-Friedrichs inequality and the norm equivalence~\eqref{eq:normEquivalence} imply that, for any $\lambda \in [0,1]$,
	\begin{align*}
		\frac{1}{2}\frac{d}{dt} \norm{\ctserr(t)}^2 + \alpha \norm{\ctserr(t)}^2 + \lambda \triplenorm{\ctserr(t)}^2 
		\leq 
		\frac{1}{2}\frac{d}{dt} \norm{\ctserr(t)}^2 + \triplenorm{\ctserr(t)}^2,
	\end{align*}
	where $\alpha \equiv \alpha_\lambda = \frac{2(1 - \lambda) }{(\Cequiv\Cpf)^2}$ (as in Definition~\ref{def:timeAccumulations}; we omit the subscript here for brevity), and therefore
	\begin{align*}
		\frac{1}{2}\frac{d}{dt} \Big( e^{\alpha t} \norm{\ctserr(t)}^2 \Big) + \lambda e^{\alpha t} \triplenorm{\ctserr(t)}^2 \leq 
			e^{\alpha t} 
			\Big(
				\spaceest{CN}(t) 
				+
				\dataest{T, CN}(t) 
			\Big) \norm{\ctserr(t)}
			+
			e^{\alpha t} 
			\Big(
				\timeest{CN}(t) 
				+
				\dataest{S, CN}(t) 
			\Big) \triplenorm{\ctserr(t)}.
	\end{align*}
	Integrating over $s \in [0, r]$, 
	we thus obtain
	\begin{align*}
		\frac{1}{2}\norm{\ctserr(r)}^2 + \lambda \int_{0}^{r} e^{\alpha(s-r)} \triplenorm{\ctserr(s)}^2 \differential s
		\leq 
		\frac{1}{2} \norm{\ctserr(0)}^2 
		+
		\int_{0}^{r} 
		&e^{\alpha (s-r)} 
		\Big(
			\spaceest{CN}(s) 
			+
			\dataest{T, CN}(s) 
		\Big) \norm{\ctserr(s)}
		\\&\quad
		+
		e^{\alpha (s-r)} 
		\Big(
			\timeest{CN}(s) 
			+
			\dataest{S, CN}(s) 
		\Big) \triplenorm{\ctserr(s)}
		\differential s.
	\end{align*}
	Invoking Lemma~\ref{lem:timeAccumulations} with $\nu = \ctserr$ on each term of the integral, we thus obtain
	\begin{align*}
		&\frac{1}{2} \norm{\ctserr(r)}^2 
		 + 
		 \lambda
		 \int_{0}^{r}  e^{\alpha(s-r)} \triplenorm{\ctserr(s)}^2 \differential s
		\leq 
		\\&\qquad
		\frac{1}{2} \norm{\ctserr(0)}^2 
		+
		\max_{s \in [0,r]} \norm{\ctserr(s)} 
			\Big(
			\min_{p \in [1,\infty]} 
			c_{p,r}
			\norm{\mathcal{S}_{CN}}_{L^{p}(0,r)}
			+ 
			\min_{p \in [1,\infty]} 
			c_{p,r}
			\norm{ \dataest{T, CN}(s) }_{L^p(0,r)}
			\Big)
		\\&\qquad
			+ 
			\Big(\int_{0}^{r} e^{\alpha(s-r)} \triplenorm{\ctserr(s)}^2 \differential s\Big)^{1/2}
			\Big(
			\min_{p \in [2,\infty]} 
			(c_{\frac{p}{2},r})^{1/2}
			\norm{\timeest{CN}}_{L^p(0,r)}
			+ 
			\min_{p \in [2,\infty]} 
			(c_{\frac{p}{2},r})^{1/2}
			\norm{\dataest{S, CN}}_{L^p(0,r)}
			\Big).
	\end{align*}
	We next apply Young's inequality, which states that if $a,b \geq 0$, then $ab \leq \frac{\epsilon}{2} a^2 + \frac{1}{2\epsilon} b^2$ for any $\epsilon > 0$.
	Picking $\epsilon = 2\lambda$ applying the inequality with $a$ and $b$ as the final two factors on the second line of the bound, we obtain
	\begin{align*}
		&\frac{1}{2} \norm{\ctserr(r)}^2 
		\leq 
		\frac{1}{2} \norm{\ctserr(0)}^2 
		+
		\max_{s \in [0,r]} \norm{\ctserr(s)} 
			\Big(
			\min_{p \in [1,\infty]} 
			c_{p,r}
			\norm{\mathcal{S}_{CN}}_{L^{p}(0,r)}
			+ 
			\min_{p \in [1,\infty]} 
			c_{p,r}
			\norm{ \dataest{T, CN}(s) }_{L^p(0,r)}
			\Big)
		\\&\qquad\qquad\qquad
			+ 
			\frac{1}{4\lambda}
			\Big(
			\min_{p \in [2,\infty]} 
			(c_{\frac{p}{2},r})^{1/2}
			\norm{\timeest{CN}}_{L^p(0,r)}
			+ 
			\min_{p \in [2,\infty]} 
			(c_{\frac{p}{2},r})^{1/2}
			\norm{\dataest{S, CN}}_{L^p(0,r)}
			\Big)^2.
	\end{align*}
	Since the right hand side of this bound is a non-decreasing function of $r$, it also provides a bound on $\max_{s \in [0,r]} \norm{\ctserr(s)}^2$.
	The result then follows by applying Young's inequality once again, taking square roots, and selecting $\lambda = \frac{1}{4}$.
\end{proof}

We note that it is necessary to apply H\"older's inequality (in the guise of Lemma~\ref{lem:timeAccumulations}) in the proof of Lemma~\ref{lem:reconstructionError} before one may derive a bound for $\max_{s\in[0,r]}\norm{\rho(s)}$.
This is because a quantity of the form
\begin{align*}
	G(r) = \int_0^r e^{\alpha_\lambda (s-r)} \abs{g(s)} \differential s
\end{align*}
can be a \emph{decreasing} function of $r$, since the exponential weight in the integral possesses a `forgetfulness' property: for fixed $s$ the function $e^{\alpha_\lambda (s-r)}$ decreases as $r$ increases.
Hence, as $r$ increases, less weight is assigned to the early part of the integral, and $G(r)$ can therefore decrease.
This means that an estimate of the form of $G$ is well suited for a final time estimate (cf.~\citet[\S 6]{Lakkis:2007du}), although not for a maximum norm estimate.
Upon applying H\"older's inequality, however, this problem is alleviated since the exponential weight and the function $g$ are integrated separately.

\subsection{Final error estimates}\label{sec:finalEstimates}
Finally, we are in a position to provide optimal order estimates for the $L^{\infty}(L^2)$ error for the backward Euler and Crank-Nicolson methods.
These are based on using the splittings~\eqref{eq:be:splitting} and~\eqref{eq:cn:splitting} of the error, and applying the estimates of Lemmas~\ref{lem:ellipreconBound} and~\ref{lem:reconstructionError} for each component.

\begin{theorem}[$L^{\infty}(0,T; L^2(\domain))$ error estimates]\label{thm:finalestimate}
Let $r \in [0,T]$.
If $\{\U[\currstep]\}_{\currstep=0}^{N}$ denotes the backward Euler approximations of $u$ at the time nodes, then
\begin{align*}
	\norm{u - U}_{L^{\infty}(0,r; L^2(\domain))}
		&\leq 
			\norm{u_0 - \interp u_0} 
			+ 
			\norm{\ellipreconest{BE}}_{L^{\infty}(0,r)}
		\\&\qquad
			+
			\sqrt{2} 
			\Big(
			\min_{p \in [1,\infty]} 
			c_{p,r}
			\norm{\spaceest{BE}}_{L^{p}(0,r)}
			+ 
			\min_{p \in [1,\infty]} 
			c_{p,r}
			\norm{\timeest{BE}}_{L^p(0,r)}
		\\&\qquad\qquad\qquad\qquad
			+ 
			\min_{p \in [1,\infty]} 
			c_{p,r}
			\norm{\dataest{T, BE}}_{L^p(0,r)}
			+ 
			\min_{p \in [2,\infty]} 
			(c_{\frac{p}{2},r})^{1/2}
			\norm{\dataest{S, BE}}_{L^p(0,r)}
			\Big),
\end{align*}
where the individual terms are defined in Definition~\ref{def:be:estimatorterms}.

If $\{\U[\currstep]\}_{\currstep=0}^{N}$ denotes the Crank-Nicolson approximations of $u$ at the time nodes, then
\begin{align*}
	\norm{u - U}_{L^{\infty}(0,r; L^2(\domain))}
		&\leq 
			\norm{u_0 - \interp u_0} 
			+ 
			\norm{\ellipreconest{CN}}_{L^{\infty}(0,r)}
			+ 
			\norm{\timereconest{CN}}_{L^{\infty}(0,r)}
		\\&\qquad
			+
			\sqrt{2}
			\Big(
			\min_{p \in [1,\infty]} 
			c_{p,r}
			\norm{\spaceest{CN}}_{L^{p}(0,r)}
			+ 
			\min_{p \in [2,\infty]} 
			(c_{\frac{p}{2},r})^{1/2}
			\norm{\timeest{CN}}_{L^p(0,r)}
		\\&\qquad\qquad\qquad\qquad
			+ 
			\min_{p \in [1,\infty]} 
			c_{p,r}
			\norm{\dataest{T, CN}}_{L^p(0,r)}
			+ 
			\min_{p \in [2,\infty]} 
			(c_{\frac{p}{2},r})^{1/2}
			\norm{\dataest{S, CN}}_{L^p(0,r)}
			\Big),
	\end{align*}
where the individual terms are defined in Definition~\ref{def:cn:estimatorterms}.
\end{theorem}
\begin{proof}
	For the backward Euler scheme, we apply the triangle inequality to observe that
	\begin{align*}
		\norm{u - \lintimerecon}_{L^{\infty}(0,r; L^2(\domain))} \leq \norm{u - \linspacetimerecon}_{L^{\infty}(0,r; L^2(\domain))} + \norm{ \linspacetimerecon - \lintimerecon}_{L^{\infty}(0,r; L^2(\domain))}.
	\end{align*}
	These two terms are then bounded using Lemmas~\ref{lem:ellipreconBound} and~\ref{lem:reconstructionError} respectively to provide the desired result.
	
	Similarly, for the Crank-Nicolson scheme, we find that
	\begin{align*}
		\norm{u - \lintimerecon}_{L^{\infty}(0,r; L^2(\domain))} \leq \norm{u - \quadspacetimerecon}_{L^{\infty}(0,r; L^2(\domain))} + \norm{ \quadspacetimerecon - \quadtimerecon}_{L^{\infty}(0,r; L^2(\domain))} + \norm{ \quadtimerecon - \lintimerecon}_{L^{\infty}(0,r; L^2(\domain))}.
	\end{align*}
	The first and second terms are once again bounded using Lemmas~\ref{lem:ellipreconBound} and~\ref{lem:reconstructionError} respectively, and we observe that, for $t \in (t^{\prevstep}, t^{\currstep}]$,
	\begin{align*}
		Q(t) - U(t) = - \frac{\timestep^2}{2} \lcurr(t) \lprev(t) \timederivh (\diffoph[\currstep] \U[\currstep] + \forcehbc[\currstep]),
	\end{align*}
	so that $\norm{Q - U} \leq \timereconest{CN}$, and the result follows.
\end{proof}

\section{Numerical experiments}
\label{sec:numerics}

\begin{wrapfigure}[41]{r}{0.45\textwidth}%
  \centering{%
  \subcaptionbox{$Q^{\currstep} = 1$ for all $\currstep$.\label{fig:weightedAccumulations:Q1}}[0.45\textwidth]{\includegraphics[width=0.45\textwidth]{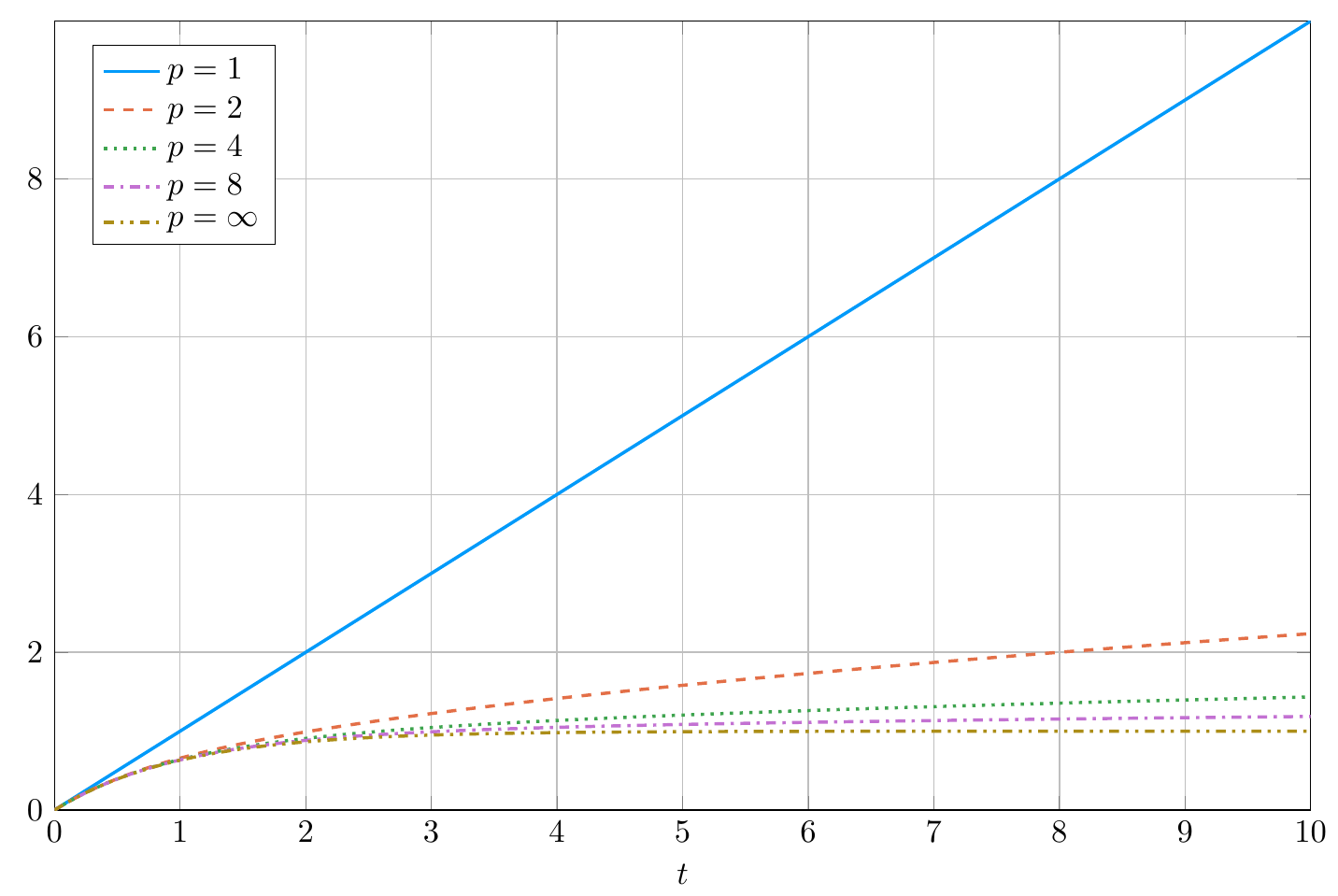}\vspace{-0.25cm}}
  
  \vspace{0.3cm}
  \subcaptionbox{$Q^{\currstep}$ specified randomly in $[0,10]$.\label{fig:weightedAccumulations:Qrand}}[0.45\textwidth]{\includegraphics[width=0.45\textwidth]{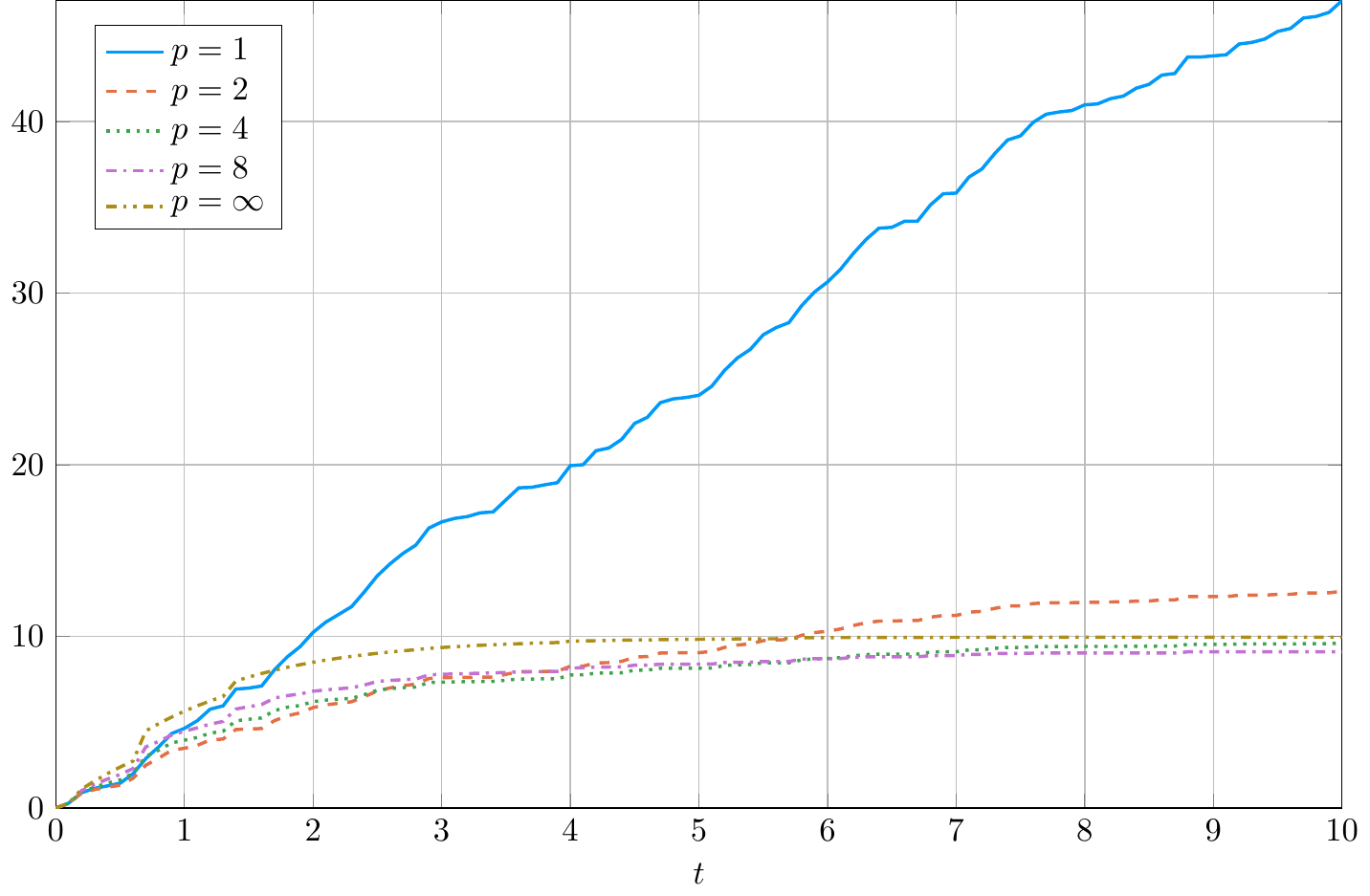}\vspace{-0.25cm}}
  
  \vspace{0.3cm}
  \subcaptionbox{$Q^{\currstep}$ as above, but with $Q^1 = 30$.\label{fig:weightedAccumulations:large_initial_Q}}[0.45\textwidth]{\includegraphics[width=0.45\textwidth]{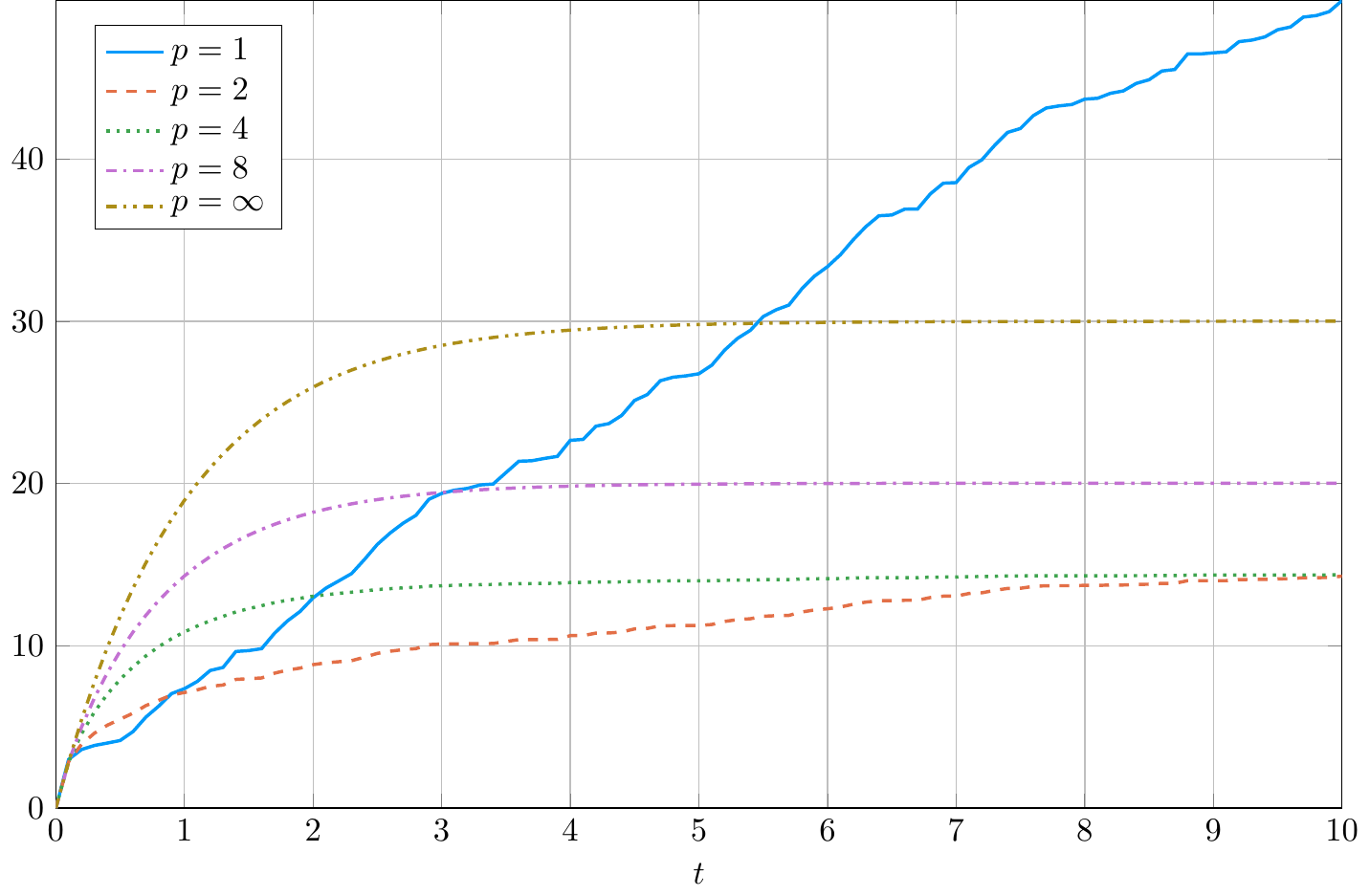}\vspace{-0.25cm}}
}
  \caption{The behaviour of the accumulations $c_{p,t} \norm{F}_{L^p(0,t)}$ for different choices of $Q^{\currstep}$ and various indicative values of $p$, with $\timestep = 0.1$.}
  \label{fig:weightedAccumulations}
\end{wrapfigure}%
In this section, we present some numerical experiments demonstrating the practical behaviour of the estimators.
We begin with an investigation into the different types of time accumulations, before considering approximating solutions to certain benchmark PDEs.
To produce results presented here, the methods and estimators described above were implemented using the open source \texttt{deal.II} finite element library~\citep{Bangerth:2007ep}.

\subsection{Comparison of accumulations}
\label{sec:accumulationComparison}
First, consider the accumulation of a term $F$ such that $F^{\currstep} = 1$ for all $\currstep$.
In this case, for $p \in [1,\infty)$, we have
\[\displaywidth=\parshapelength\numexpr\prevgraf+2\relax
    \norm{F}_{L^p(0,t^m)} = \Big( \sum_{\currstep = 1}^{m} \timestep (F^{\currstep})^{p} \Big)^{1/p} = \Big( \sum_{\currstep = 1}^{m} \timestep \Big)^{1/p} = (t^{m})^{1/p},
    \]
and $\norm{F}_{L^\infty(0,t^m)} = 1$.
It is clear, therefore, that for finite values of $p$ the accumulation will grow unboundedly, while for $p = \infty$ the value will remain constant.
Despite this, for $t^m < 1$, the choice $p=1$ clearly provides the smallest accumulation.
Thus, we find that
\[\displaywidth=\parshapelength\numexpr\prevgraf+2\relax
    \min_{p \in [1,\infty]} \norm{F}_{L^p(0,t^m)} =
    \begin{cases}
        t^{m} &\text{ when } t^m \leq 1,
        \\
        1 &\text{otherwise},
    \end{cases}
    \]
implying that the $L^1$-type accumulation is best for short time computations (i.e. with $T \leq 1$), while the $L^{\infty}$-type accumulation is best for long time computations.

The situation is different, however, if we include the accumulation control coefficients $c_{p,t}$, introduced in Definition~\ref{def:timeAccumulations}.
In this case, the growth of the different weighted accumulations $c_{p,t} \norm{F}_{L^p(0,t)}$ is depicted in Figure~\ref{fig:weightedAccumulations}, under the supposition that the constant $\alpha_{\lambda} = 1$.
Figure~\ref{fig:weightedAccumulations:Q1} shows the growth of the different accumulations in the case when $F^{\currstep} = 1$ for all $\currstep$, and in this case it appears that the control coefficients have a levelling effect, in that the $L^{\infty}$-type accumulation is always the smallest.
This is no longer true in the more practically relevant situation when $F^{\currstep}$ varies randomly in the interval $[0,10]$, depicted in Figure~\ref{fig:weightedAccumulations:Qrand}.
In this case, it is still clear that the $L^1$-type accumulation grows much larger than the others, while the $L^{\infty}$-type accumulation grows the slowest.
The actual minimum, however, appears to be somewhere in between and changes throughout the time interval.
For much longer times, however, it is expected that the minimum will ultimately be the $L^{\infty}$-type accumulation since in this case this is known to be bounded above by $10$.

Figure~\ref{fig:weightedAccumulations:large_initial_Q} shows the accumulations applied to the same data set but with $F^1 = 30$.
This makes little difference to the $L^1$-type accumulation, although has a dramatic effect for larger values of $p$.
In such a situation, therefore, the $L^{\infty}$-type accumulation will not be minimal until $t$ becomes very large.
Instead, in this situation it may be preferable to use a combination of different accumulation types on different subintervals of the time domain, as described in Section~\ref{sec:timeAccumulations}.

\subsection{Implementing $L^p$-type accumulations}\label{sec:implementation}
To evaluate the estimator using the optimal accumulation for each term, we note that it is possible to store the contributions to each term from each time step, and then implement a numerical optimisation procedure to systematically find the optimal value of p.
However, such an algorithm would be expensive in terms of both memory (to store the data from every time step) and computation (to find the minimal accumulation).
Instead, we propose to evaluate the estimator using just a small number of accumulations for each term, for instance for $p \in \{ 1,2,4,8,16,\infty \}$.
Computing the estimator like this only requires storing six numbers, as the degree $p$ accumulation can be updated on each step using the update rule
\begin{align*}
	\norm{F}_{L^p(0,t^{\currstep})} = 
	\begin{cases}
		\big( \norm{F}_{L^p(0,t^{\prevstep})}^p + \timestep \,\, (F^{\currstep})^p \big)^{1/p} 
		&\text{ for } p \in [1,\infty)
		\\
		\max \{ \norm{F}_{L^p(0,t^{\prevstep})}, F^{\currstep} \} 
		&\text{ for } p = \infty.
	\end{cases}
\end{align*}
This is the approach taken for evaluating the estimator in all of the numerical experiments presented below.

\subsection{Benchmark numerical examples}
We now investigate the behaviour of the estimator when applied to a sequence of example problems.
For these examples we take $A$ to be the $\spacedim \times \spacedim$ identity matrix and $\mu = 0$.
The results are computed over the domain $\domain = (0,1)\times(0,1)$ with $\finaltime=15$, using a sequence of progressively finer meshes consisting of uniform square elements and polynomial degree $\degree = 1$.
Unless otherwise stated, the mesh sequence consists of meshes with $2^{2i}$ elements, where $i \in \{2,3,4,5,6\}$.
The mesh size may therefore be computed in each case as $h_i = 2^{1/2 - 2i}$.

For each example problem on each sequence of meshes we plot a composite figure of the resulting data.
Subfigure (A) of each demonstrates the behaviour of the $L^{\infty}(0,t; L^2(\domain))$ error and the total estimator, computed by taking the minimum of a subset of the accumulations as described in Section~\ref{sec:implementation}, with the simulation on each mesh being shown as a separate line on the plot (solid lines indicate the results computed on the finest mesh).
Beneath each of these, we plot the convergence rate with respect to $i$ as a function of time, computed for a quantity $F^i$ by
\begin{align*}
	\operatorname{rate}_i(t) = \frac{\log(F^i(t)) - \log(F^{i-1}(t))}{\log(h_i) - \log(h_{i-1})}.
\end{align*}
Since the timestep $\timestep[i]$ is a specified function of $h_{i}$ in each case, this provides a meaningful definition of convergence rate.
Next to these we plot the effectivity of the estimator, computed as
\begin{align*}
	\operatorname{effectivity}_i(t) = \frac{\operatorname{estimator}_i(t)}{\norm{u - \lintimerecon_i}_{L^{\infty}(0,t; L^2(\domain))}}.
\end{align*}
Here, $\operatorname{estimator}_i(t)$ and $\lintimerecon_i(t)$ denote the estimator and discrete solution calculated on mesh $i$.

Subfigure (B) of each plot shows the magnitudes and convergence rates of each component of the error estimator, computed using the minimal accumulation in each case.
Finally, subfigure (C) in each case shows a comparison between the estimators evaluated using various specified types of accumulation.
The line marked `$L^1$ accumulations' shows the value of the estimator when all components are evaluated using $L^p$-type accumulations with the least admissible value of $p$ in each case (for some terms this is 1, for others it is 2).
Similarly, the line `$L^2$ accumulations' shows the behaviour where $p$ is taken as value 2 for all components, and the line `$L^{\infty}$ accumulations' shows the effect of taking $p = \infty$.
To the right of these, we plot the effectivity computed for each of these three estimators.

\subsubsection{Sinusoidal benchmark}
We consider the problem with the solution
\begin{align}\label{eq:ex:sinsinsin}
	u(t,x,y) = \sin(\pi t)\sin(\pi x)\sin(\pi y).
\end{align}
For this example, we show results for the backward Euler method with time steps linked to the spatial discretisation as $\timestep[i] = h_i^2$ and $\timestep[i] = h_i$ in Figures~\ref{fig:be:sinsinsin:hsq} and~\ref{fig:be:sinsinsin:h} respectively, and for the Crank-Nicolson method with $\timestep[i] = h_i$ and $\timestep[i] = h_i^{1/2}$ in Figures~\ref{fig:cn:sinsinsin:h} and~\ref{fig:cn:sinsinsin:hroot} respectively.
Since a priori error estimates inform us that the $L^{\infty}(0,t; L^2(\domain))$ should converge with order $\mathcal{O}(\timestep + h^2)$ for the backward Euler method and with order $\mathcal{O}(\timestep^2 + h^2)$ for the Crank-Nicolson scheme, we deduce that the expected convergence rates for the four simulations should be 2, 1, 2, and 1 respectively.
For the Crank-Nicolson method with $\timestep[i] = h_i^{1/2}$, we include additional data for $i \in \{7,8\}$ in order to better demonstrate the asymptotic behaviour of the estimator.

\begin{wrapfigure}[20]{r}{0.5\textwidth}%
  \centering{%
  \includegraphics[width=0.45\textwidth]{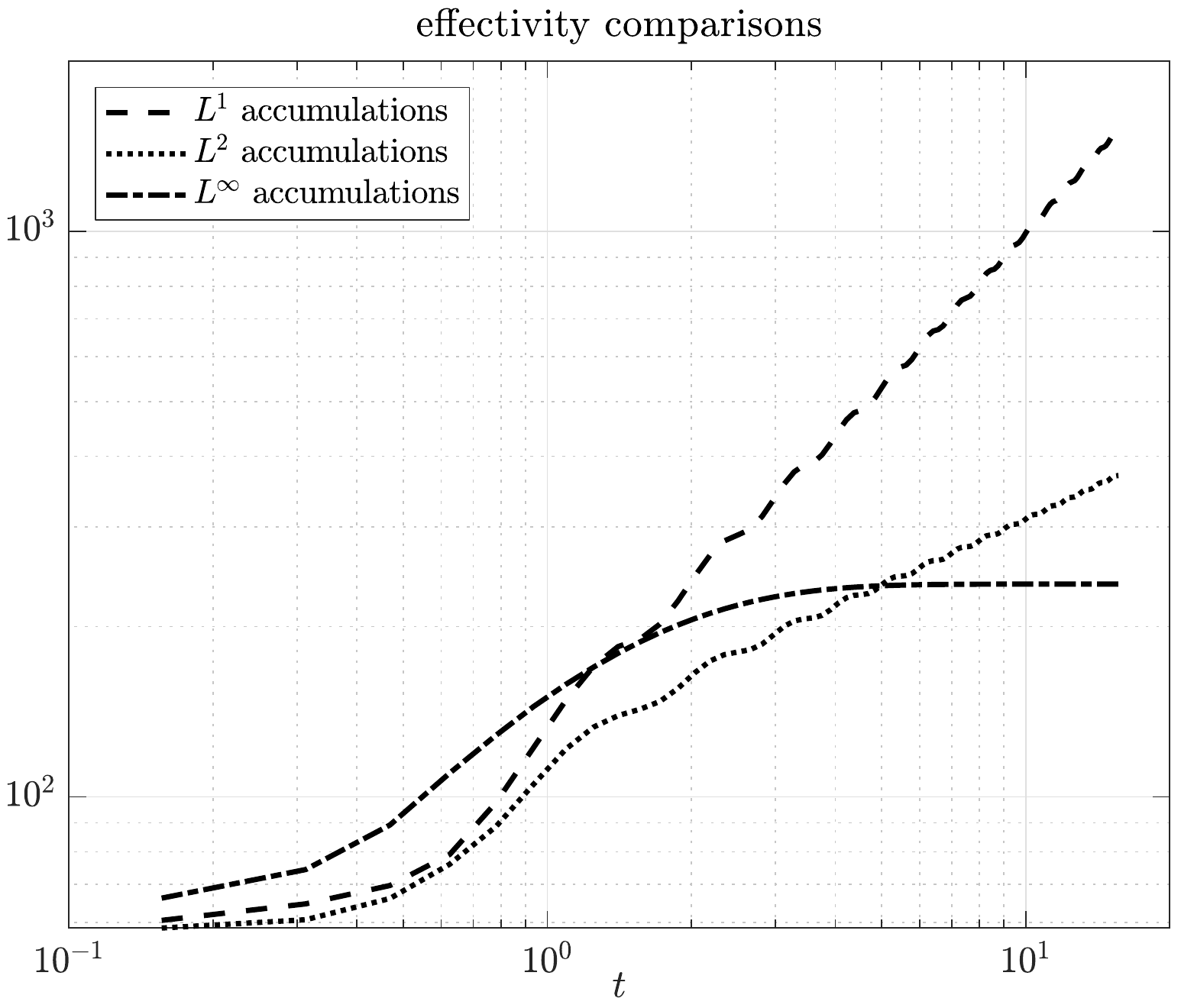}}
  \caption{The effectivity comparison of Figure~\ref{fig:cn:sinsinsin:h}(C) replotted on logarithmic axes, revealing the asymptotic rates at which the different effectivities grow.}
  \label{fig:cn:sinsinsin:loglogeff}
\end{wrapfigure}%

Examining the plots in subfigure (A) for each figure indicate that these rates are attained by the true error, while the estimator seems to converge slightly faster on the coarser meshes for both the cases when the expected rate is 1 (Figures~\ref{fig:be:sinsinsin:h} and~\ref{fig:cn:sinsinsin:hroot}).
Since this appears to be settling down towards the expected rates on the finer meshes, this quirk is presumably attributable to some pre-asymptotic behaviour, and could be due to the fact that we are ignoring the values of the constants weighting each term of the estimator, meaning that the spatial estimators (which converge more quickly) initially dominate over the temporal estimators which are responsible for restricting the total estimate to order 1.
This pre-asymptotic behaviour observed in Figures~\ref{fig:be:sinsinsin:h} and~\ref{fig:cn:sinsinsin:hroot} also has an impact on the effectivities of the estimator, we are significantly larger on coarser meshes than on finer meshes in the sequence.
In Figures~\ref{fig:be:sinsinsin:hsq} and~\ref{fig:cn:sinsinsin:h}, however, where the error is expected to converge at rate 2, the estimator converges much more similarly to the error, resulting in effectivities which do not depend so much on the mesh.

What's striking in all four figures, though, is that the effectivities become \emph{constant} with respect to time.
This is entirely due to the use of the $L^{\infty}$-type accumulations in the estimator, as shown by the plots in subfigure (C) in each case.
The estimators making use of the $L^1$ and $L^2$-type accumulations grow much faster in this case than the actual error.
This is best demonstrated by Figure~\ref{fig:cn:sinsinsin:loglogeff}, in which the data from the effectivity comparison of Figure~\ref{fig:cn:sinsinsin:h}(C) is replotted on logarithmic axes.
It becomes evident from this that, after an initial period in which the accumulation coefficient dominates the profile of each curve, the effectivity of the $L^1$-type estimator grows like $t$, the $L^2$-type estimator's effectivity grows like $t^{1/2}$, and the $L^{\infty}$-type estimator becomes constant.

\subsubsection{Polynomial benchmark}

We consider a problem with the solution
\begin{align}\label{eq:ex:poly}
	u(t,x,y) = \frac{x (x - 1) y (y - 1)}{250} t(t-2)(t-4)(t-6)(t-8)(t-10) ,
\end{align}
which satisfies the problem with right hand side
\begin{align*}
	f(t,x,y) = \frac{x (x - 1) + y (y - 1)}{125} t(t-2)(t-4)(t-6)(t-8)(t-10) .
\end{align*}
We note that for $t \in [0,10]$, $\max_{(x,y) \in \domain} u(t,x,y)$ oscillates within $[-0.3,0.1]$, but the solution grows so rapidly when $t \in [10,15]$ that $\max_{(x,y) \in \domain} u(15,x,y) = 42,230$.
The interesting features of this benchmark problem are therefore twofold: firstly, the forcing function is non-zero on $\boundary$; secondly, we are interested to see how the estimator behaves with the rapid growth in the solution.

The results obtained from the backward Euler and Crank-Nicolson methods applied to this problem are shown in Figures~\ref{fig:be:poly} and~\ref{fig:cn:poly} respectively.
To address the first point, we observe that all components of both estimators remain of at least optimal order due to the modified definition of the elliptic reconstruction operator.
The second point is more interesting.
We observe that the effectivity of both estimators are relatively well behaved, remaining around 100.
However, from the comparison of the estimators and effectivities obtained using different types of accumulation, we observe that in neither case is it optimal to use the $L^{\infty}$-type accumulations; rather, $L^2$-type accumulations seem to behave best. 
Broadly speaking, this can be attributed to the fact that the error is large early on in the simulation, as in case (C) of Section~\ref{sec:accumulationComparison}, so the $L^1$ and $L^2$-type accumulations are preferable early on, while the contribution to the error on each timestep grows for $t > 10$, ensuring that the $L^{\infty}$-type accumulations also grow.
We note that this could be a situation in which it is preferable to use different types of accumulation on different sections of the time domain, as discussed in Section~\ref{sec:timeAccumulations}.

\section{Conclusions}\label{sec:conclusion}

In summary, we have derived a family of new optimal order estimates for the $L^{\infty}(0,t; L^2(\domain)$ norm error of finite element discretisations with backward Euler and Crank-Nicolson time-stepping schemes for a class of parabolic problems.
The estimates in this family are almost all new, and allow the individual terms of the estimator to accumulate through time in an $L^{p}(0,t)$ fashion for any $p \in [1,\infty]$.
Included amongst this spectrum are, of course, previously known estimates relying on $L^1$ or $L^2$-type time accumulations of their terms, which we have demonstrated to exhibit effectivities (the ratio of the estimator to the true error) which grow like $t$ or $\sqrt{t}$, where $t$ is the simulation duration.
Estimators based on $L^{\infty}$-type accumulations were previously derived using parabolic duality-based techniques, although it was previously unknown how these could be proven using energy arguments as they have been here.
The advantage of using $L^{\infty}$-type accumulations is that the estimators attain constant effectivities on benchmark problems, meaning they are much better suited as error estimates for long time simulations and offer potential for deriving lower bounds on the error.

The technique we used for deriving these estimates is fundamentally based on the structure of the partial differential equation, rather than the construction of the numerical method.
To ensure the details of the new technique remain clear, we have only demonstrated the technique for finite element discretisations of linear parabolic problems with backward Euler and Crank-Nicolson time-stepping schemes, although we expect it will be applicable more widely --- both in the sense of wider classes of model problems involving nonlinearities for instance, and in the sense of other varieties of numerical scheme.
Similarly, the technique applies equally well in the case of adaptive space and time meshes, although presents the interesting conundrum of devising mesh refinement schemes which behave well when the mesh transfer error is accumulated in an $L^{\infty}(0,t)$ fashion.

\section*{Acknowledgements}
The author would like to thank T. Pryer for many useful comments on the manuscript.
This work was partially supported by the EPSRC grant EP/P000835/1.

%% file: parts/pagefigs.tex

\begin{figure}%
  \centering{%
  \subcaptionbox{Behaviour of the error and estimator}[\textwidth]{\includegraphics[width=0.9\textwidth]{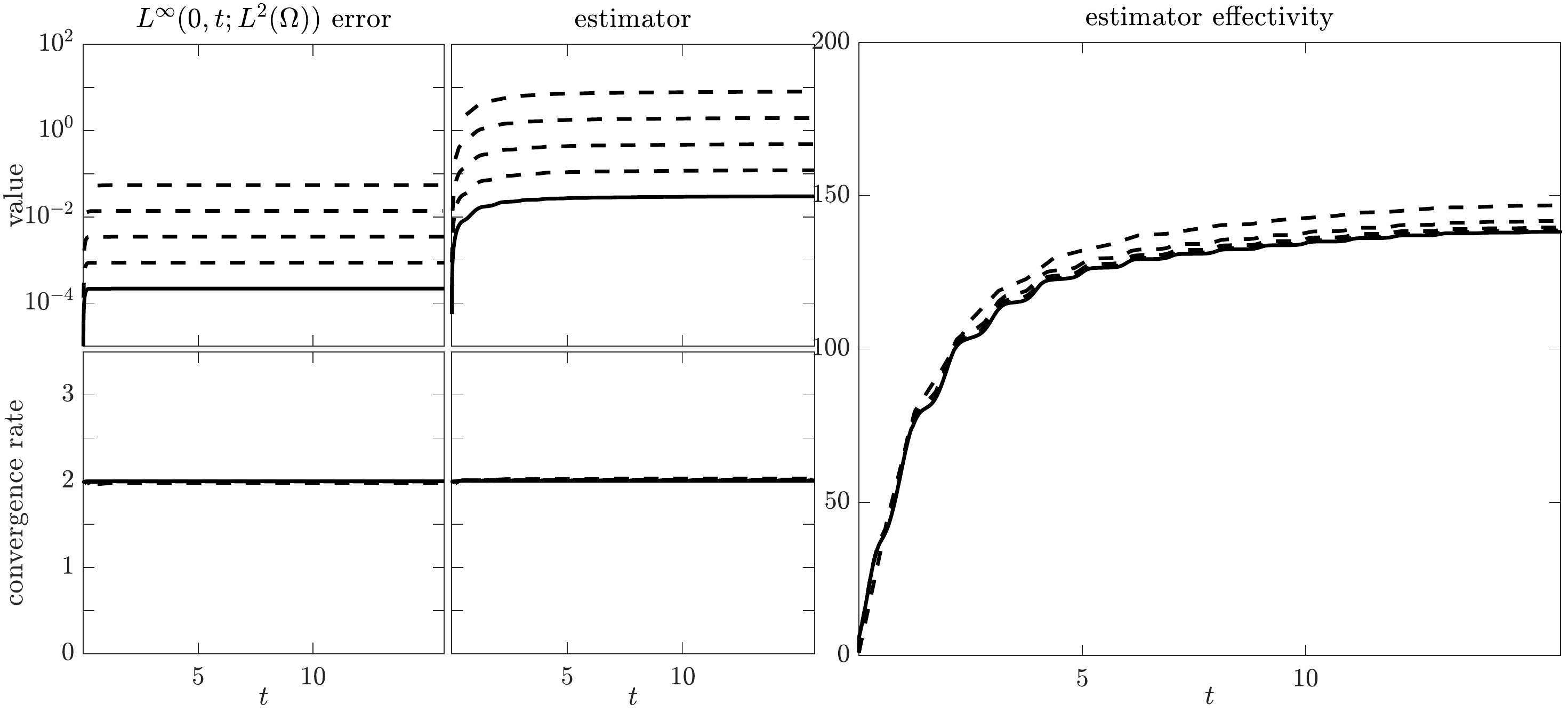}\vspace{-0.75em}}
  \vspace{1em}
  
  \subcaptionbox{Behaviour of the components of the estimator}[\textwidth]{\includegraphics[width=0.9\textwidth]{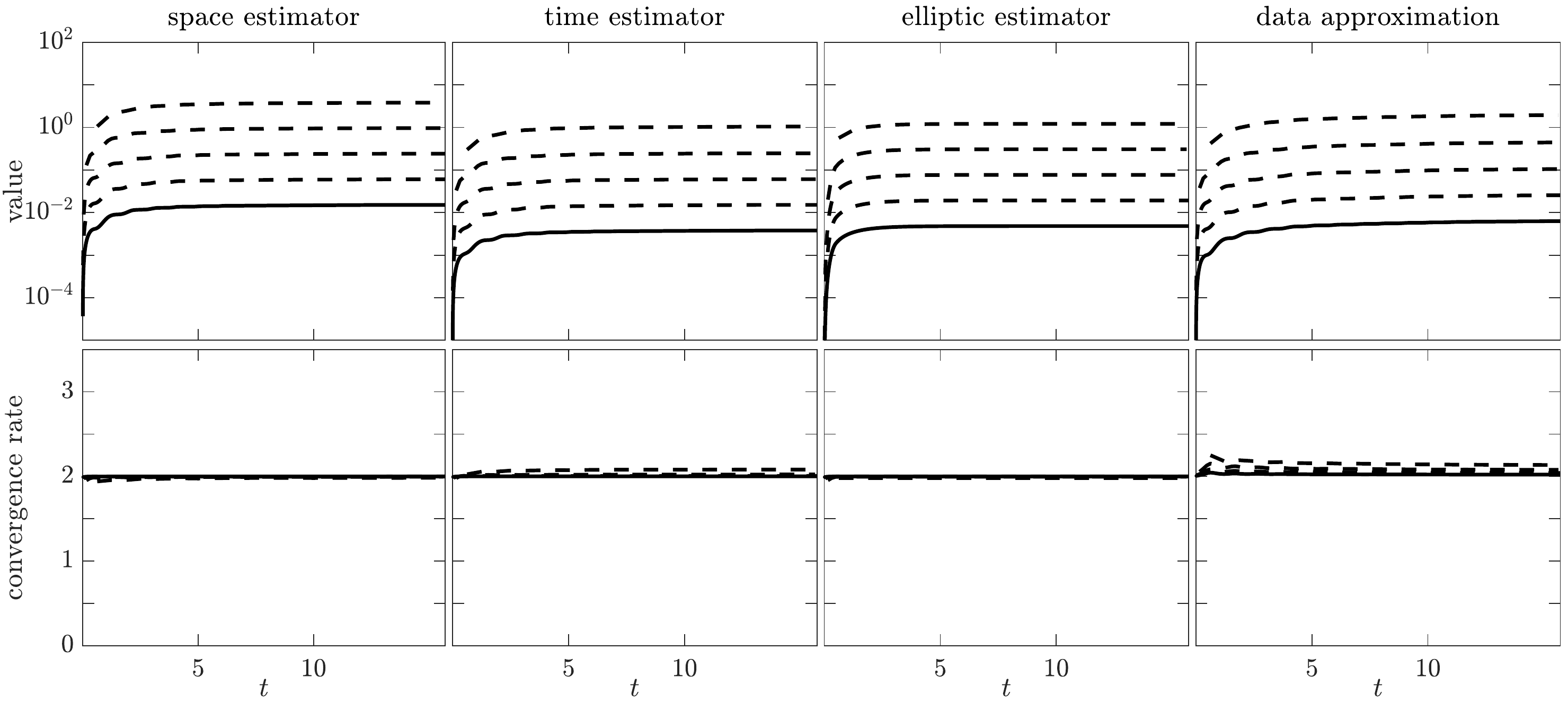}\vspace{-0.75em}}
  \vspace{1em}
  
  \subcaptionbox{Comparison of the impact of different time accumulation types on the estimator}[\textwidth]{\includegraphics[width=0.9\textwidth]{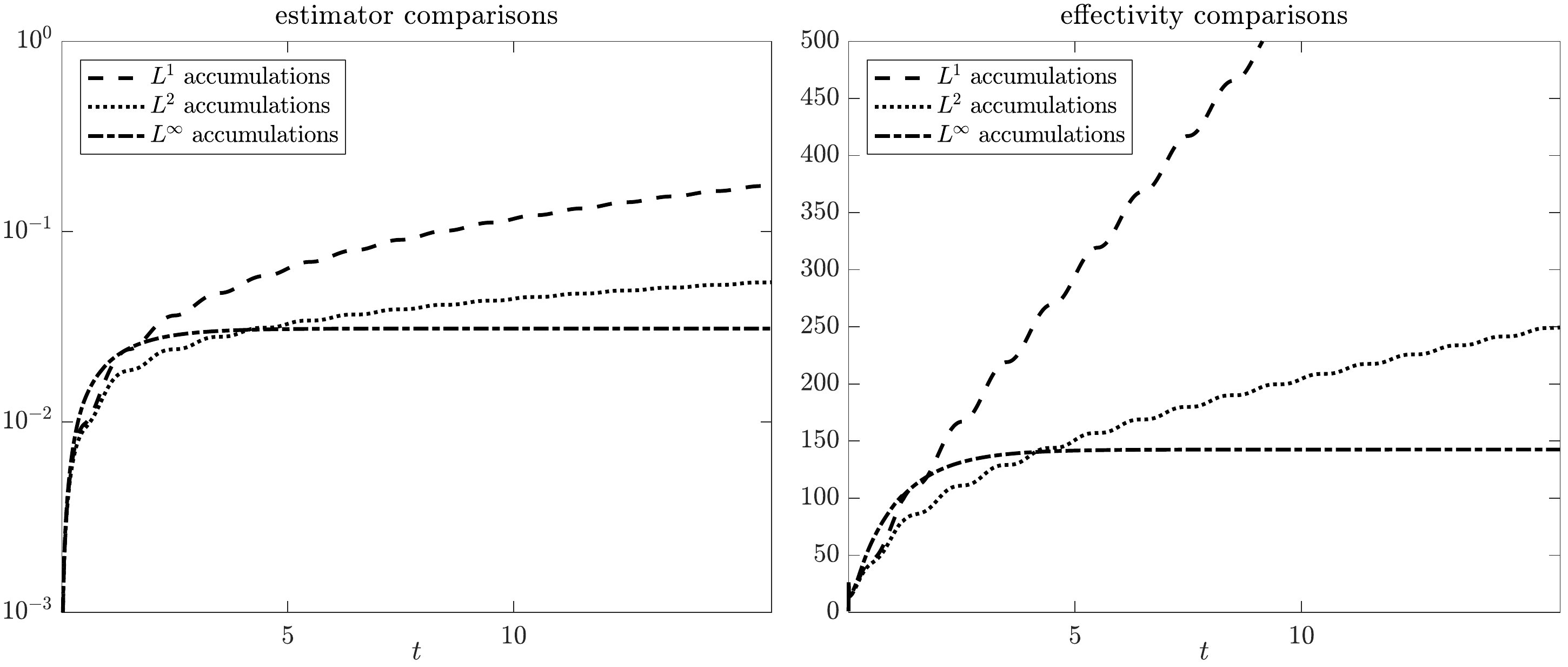}\vspace{-0.75em}}%
  }%
  \caption{Behaviour of the error and estimator for the backward Euler scheme applied to the sinusoidal example~\eqref{eq:ex:sinsinsin} with $\tau \approx h^2$.}
  \label{fig:be:sinsinsin:hsq}
\end{figure}

\begin{figure}%
  \centering{%
  \subcaptionbox{Behaviour of the error and estimator}[\textwidth]{\includegraphics[width=0.9\textwidth]{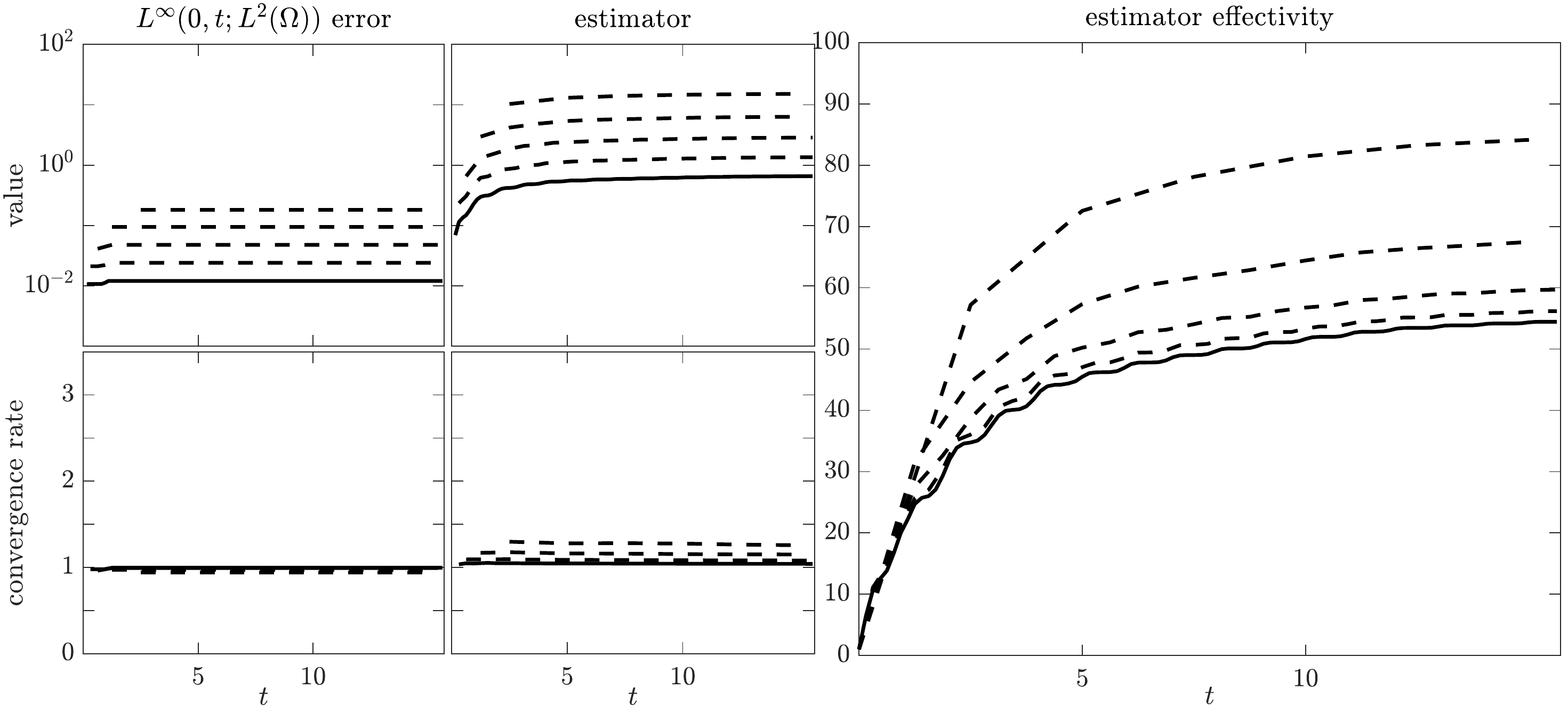}\vspace{-0.75em}}
  \vspace{1em}
  
  \subcaptionbox{Behaviour of the components of the estimator}[\textwidth]{\includegraphics[width=0.9\textwidth]{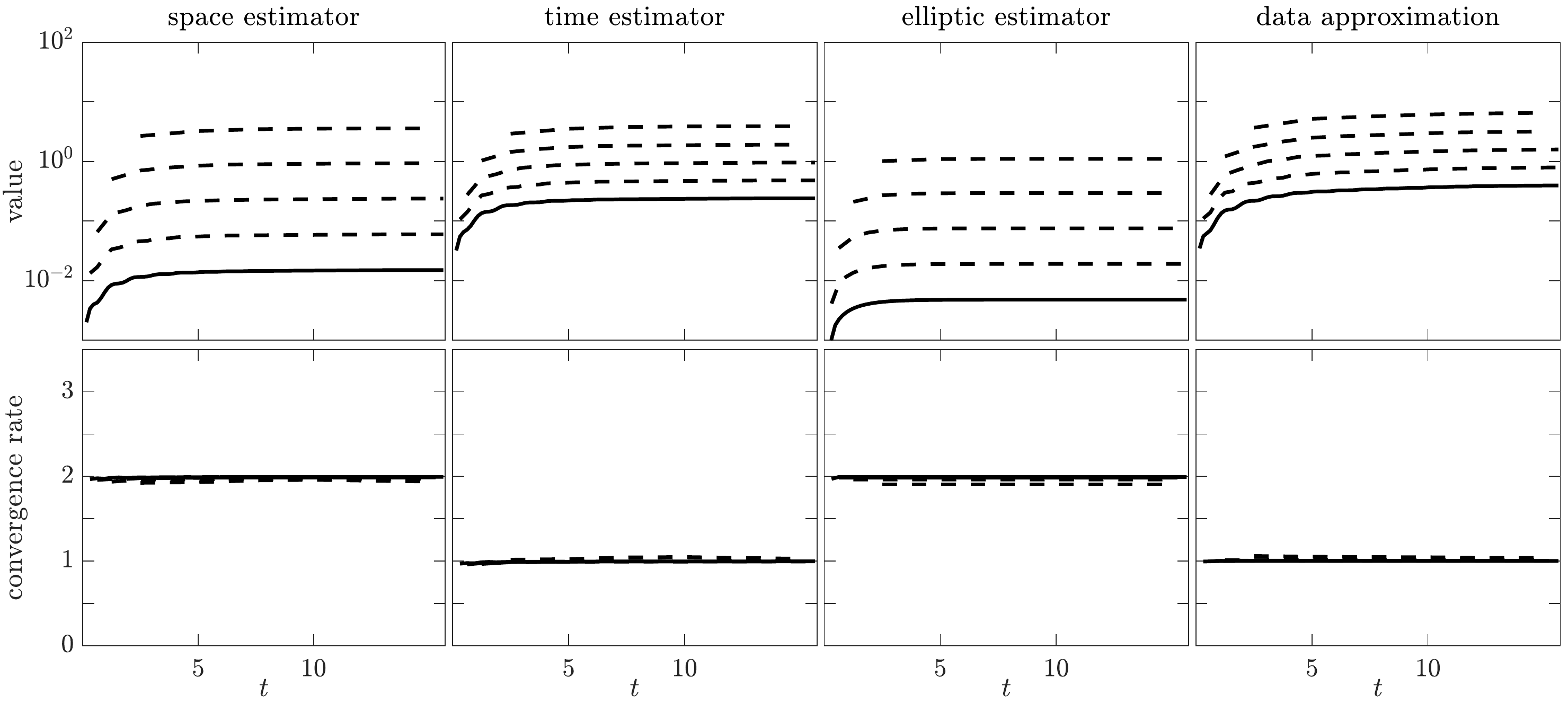}\vspace{-0.75em}}
  \vspace{1em}
  
  \subcaptionbox{Comparison of the impact of different time accumulation types on the estimator}[\textwidth]{\includegraphics[width=0.9\textwidth]{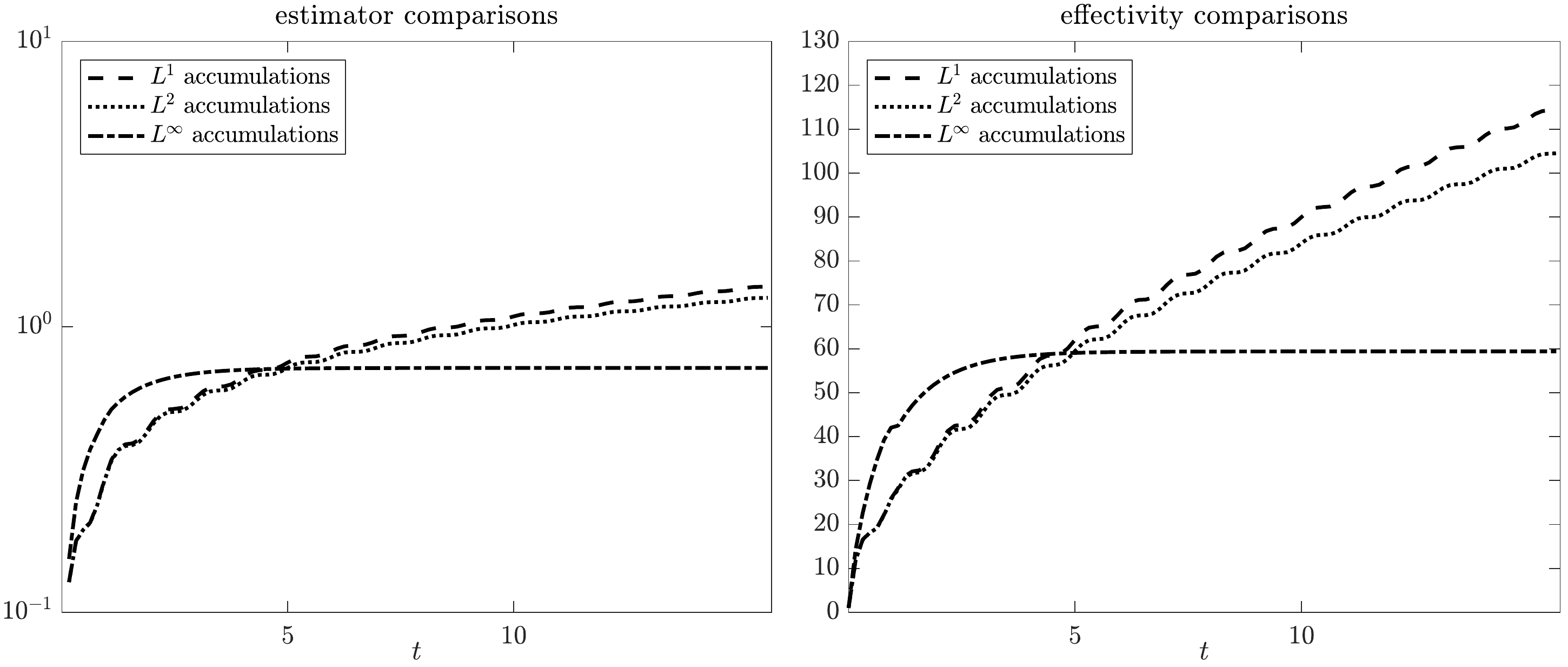}\vspace{-0.75em}}%
  }%
  \caption{Behaviour of the error and estimator for the backward Euler scheme applied to the sinusoidal example~\eqref{eq:ex:sinsinsin} with $\tau \approx h$.}
  \label{fig:be:sinsinsin:h}
\end{figure}

\begin{figure}%
  \centering{%
  \subcaptionbox{Behaviour of the error and estimator}[\textwidth]{\includegraphics[width=0.9\textwidth]{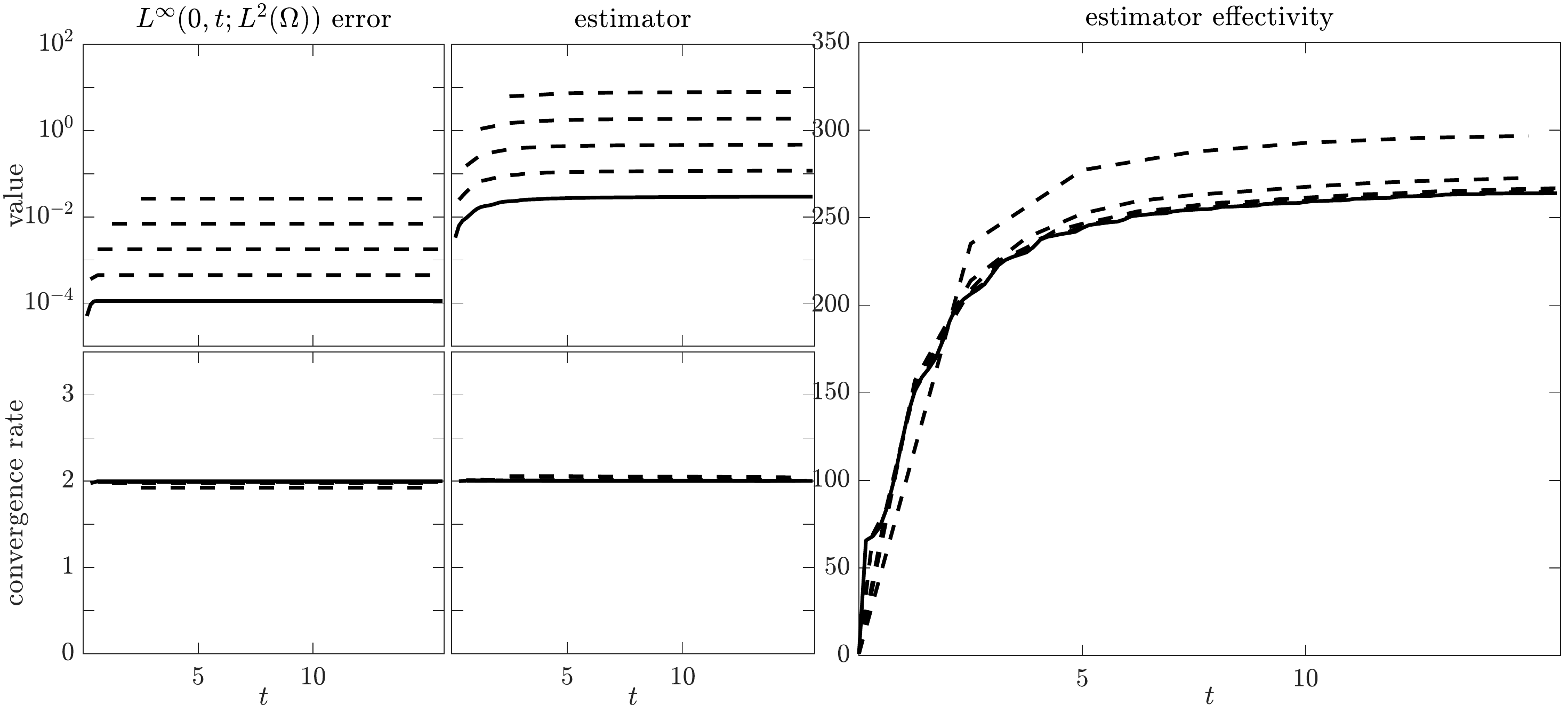}\vspace{-0.75em}}
  \vspace{1em}
  
  \subcaptionbox{Behaviour of the components of the estimator}[\textwidth]{\includegraphics[width=0.9\textwidth]{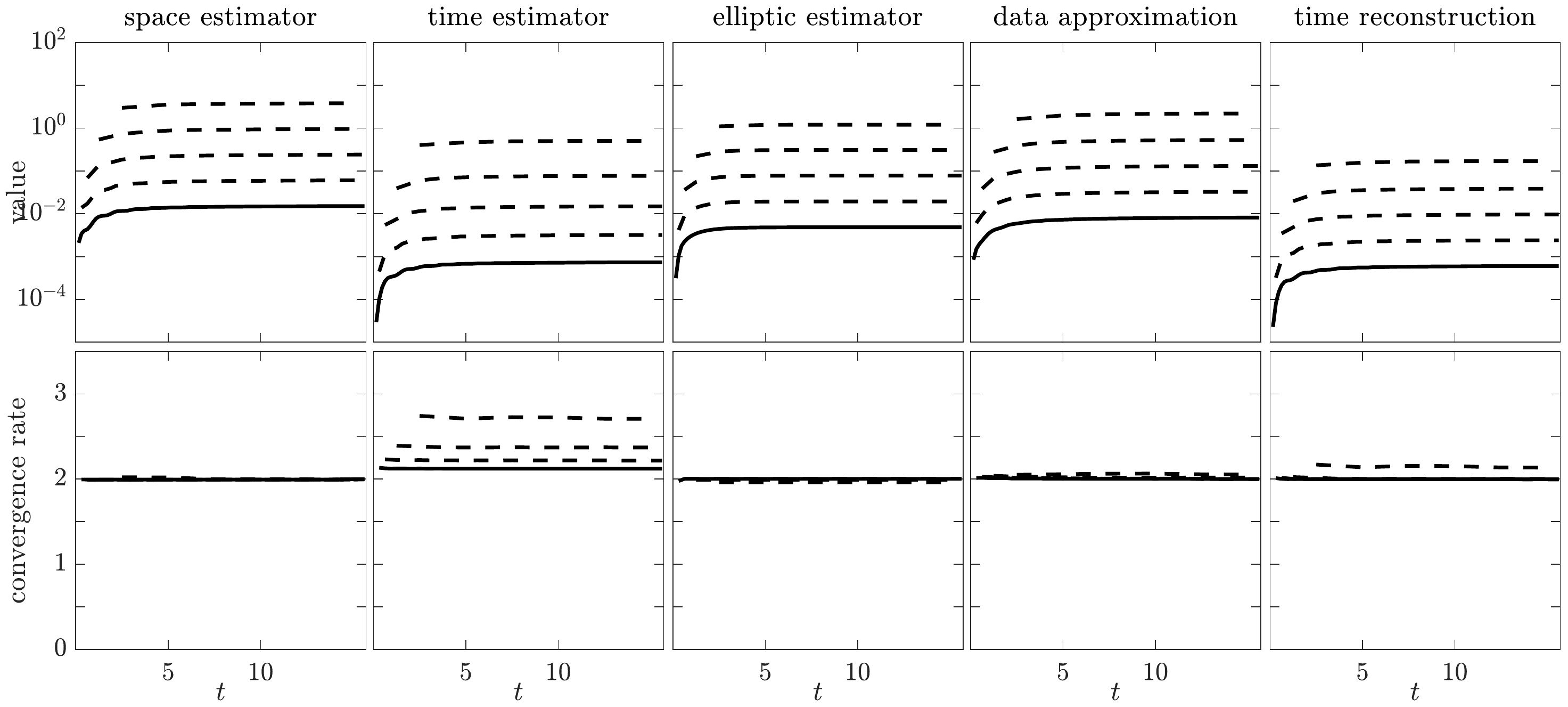}\vspace{-0.75em}}
  \vspace{1em}
  
  \subcaptionbox{Comparison of the impact of different time accumulation types on the estimator}[\textwidth]{\includegraphics[width=0.9\textwidth]{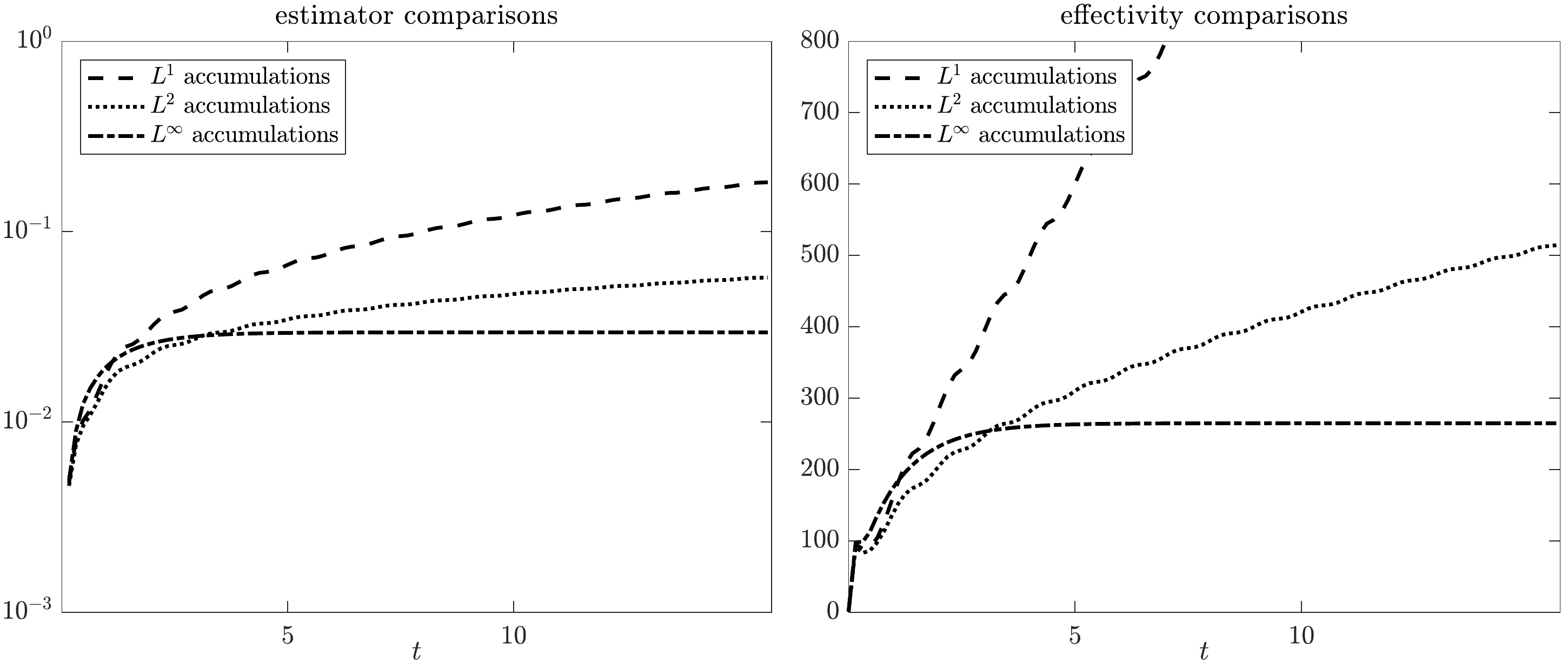}\vspace{-0.75em}}%
  }%
  \caption{Behaviour of the error and estimator for the Crank-Nicolson scheme applied to the sinusoidal example~\eqref{eq:ex:sinsinsin} with $\tau \approx h$.}
  \label{fig:cn:sinsinsin:h}
\end{figure}

\begin{figure}%
  \centering{%
  \subcaptionbox{Behaviour of the error and estimator}[\textwidth]{\includegraphics[width=0.9\textwidth]{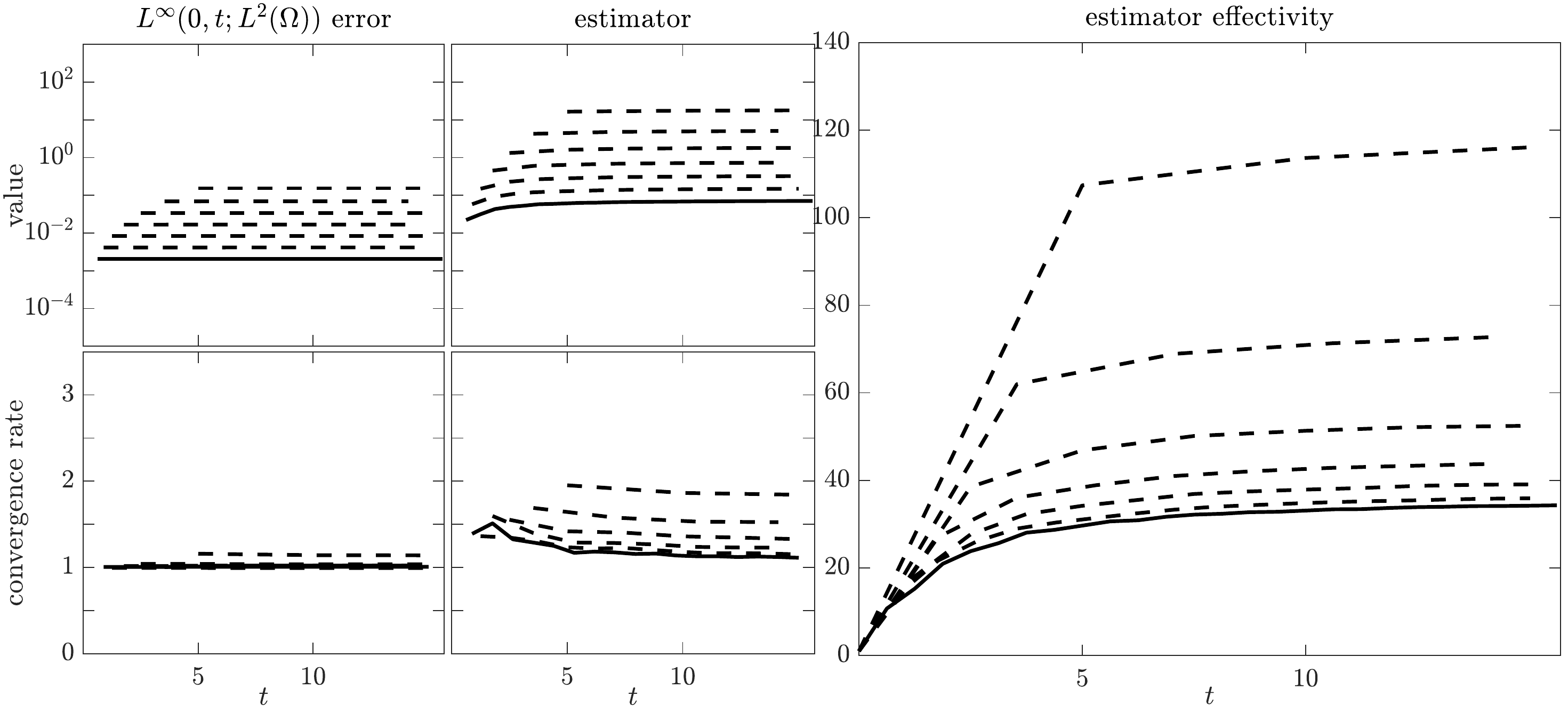}\vspace{-0.75em}}
  \vspace{1em}
  
  \subcaptionbox{Behaviour of the components of the estimator}[\textwidth]{\includegraphics[width=0.9\textwidth]{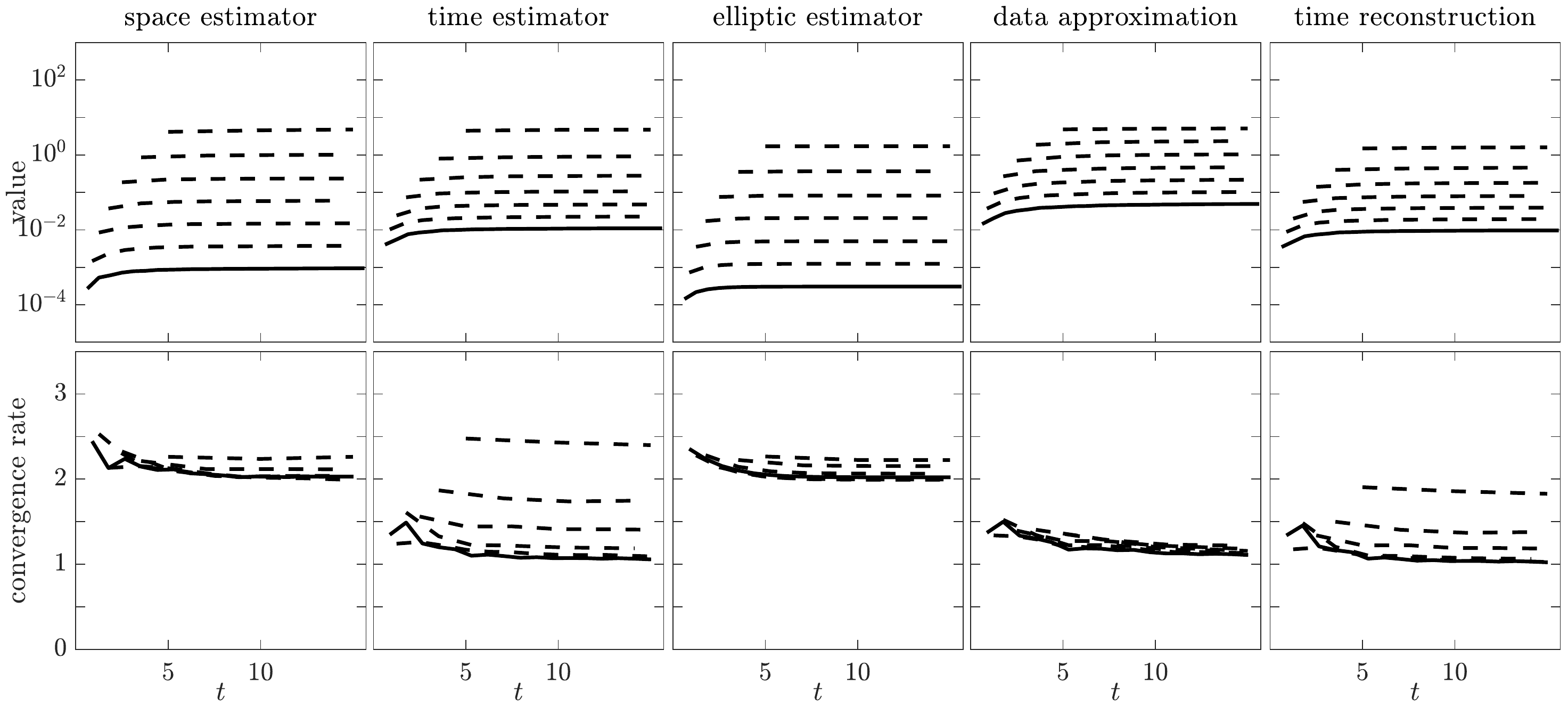}\vspace{-0.75em}}
  \vspace{1em}
  
  \subcaptionbox{Comparison of the impact of different time accumulation types on the estimator}[\textwidth]{\includegraphics[width=0.9\textwidth]{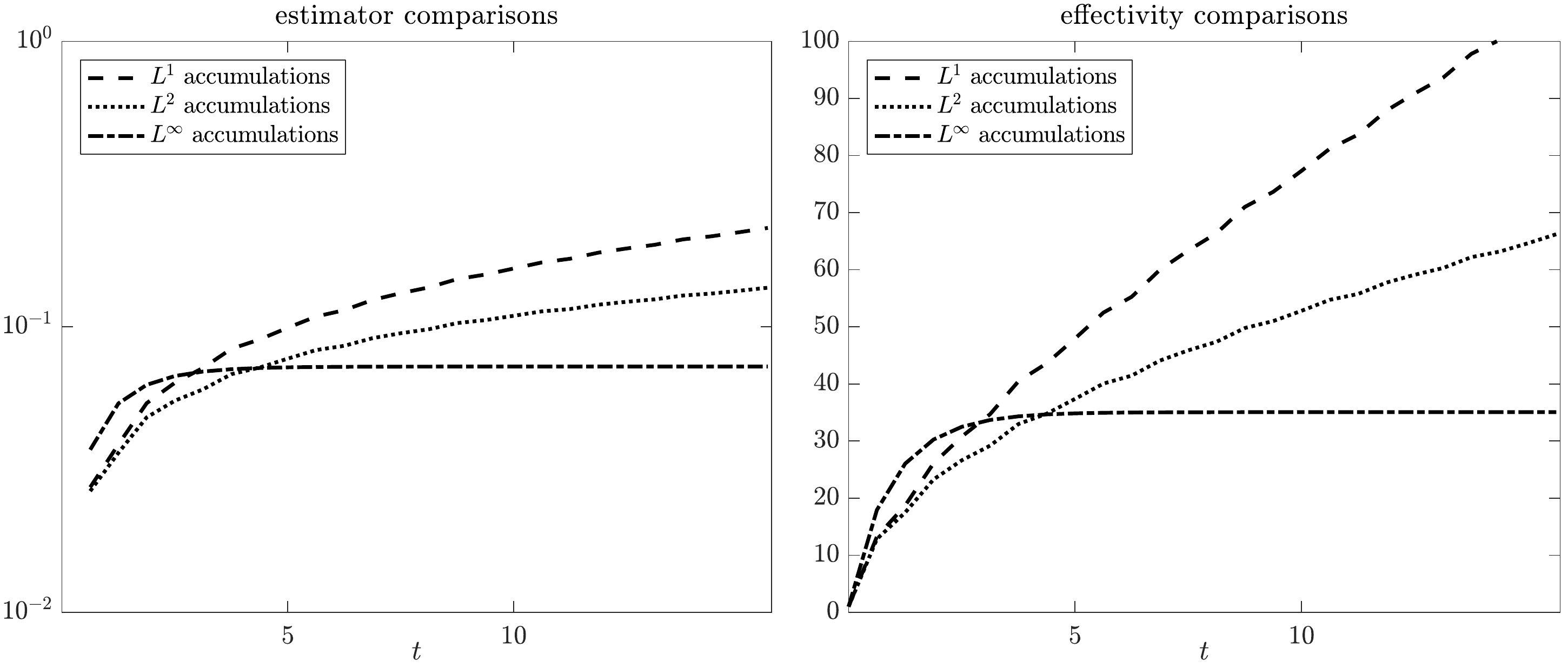}\vspace{-0.75em}}%
  }%
  \caption{Behaviour of the error and estimator for the Crank-Nicolson scheme applied to the sinusoidal example~\eqref{eq:ex:sinsinsin} with $\tau \approx \sqrt{h}$.}
  \label{fig:cn:sinsinsin:hroot}
\end{figure}

\begin{figure}%
  \centering{%
  \subcaptionbox{Behaviour of the error and estimator}[\textwidth]{\includegraphics[width=0.9\textwidth]{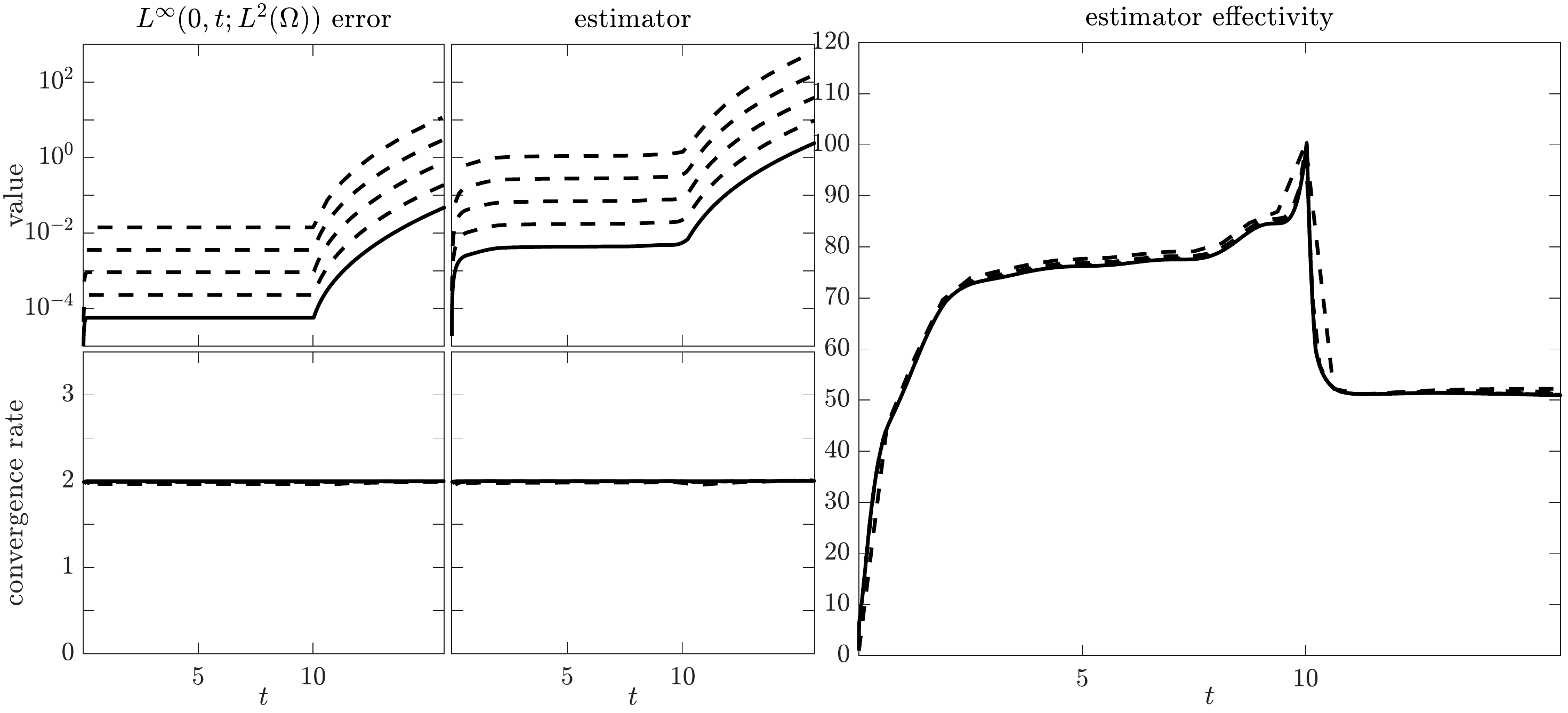}\vspace{-0.75em}}
  \vspace{1em}
  
  \subcaptionbox{Behaviour of the components of the estimator}[\textwidth]{\includegraphics[width=0.9\textwidth]{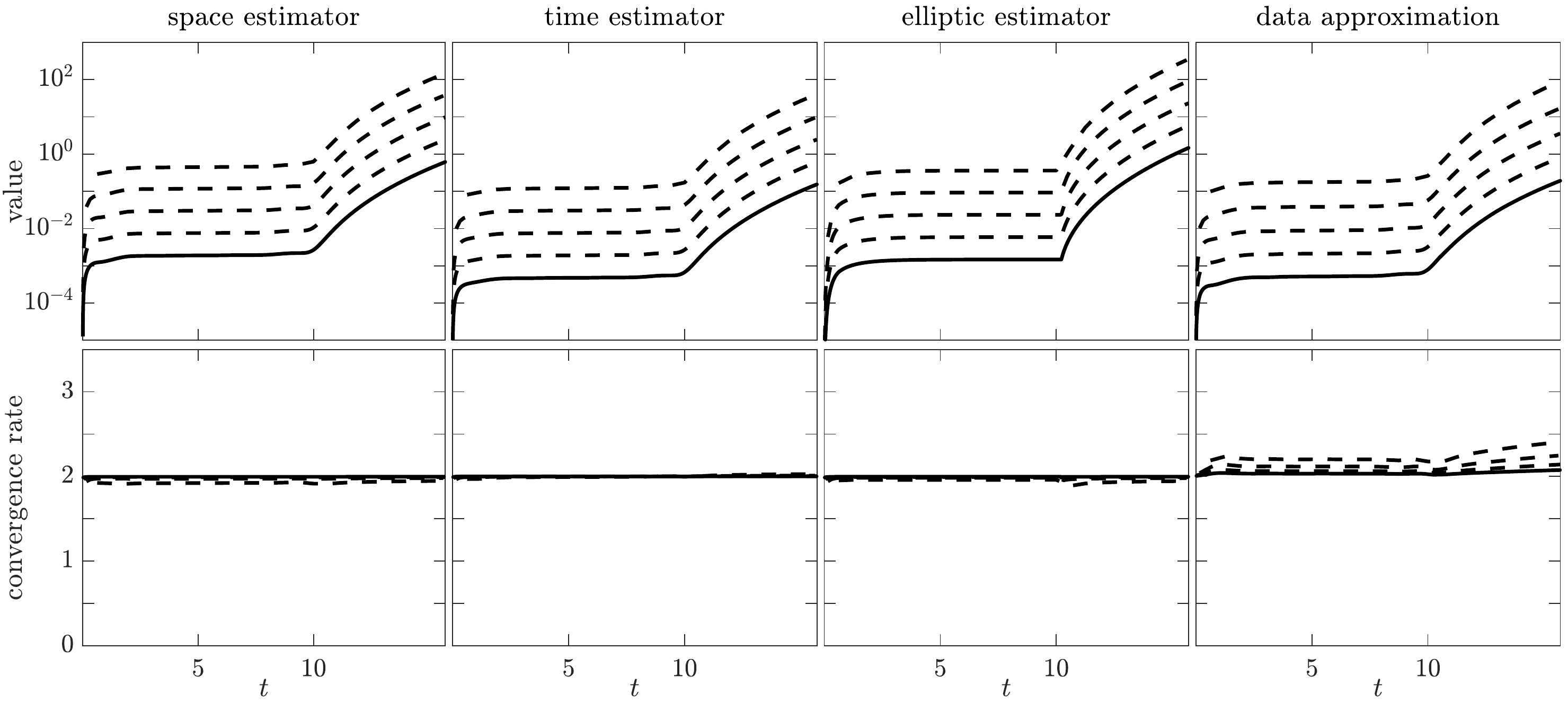}\vspace{-0.75em}}
  \vspace{1em}
  
  \subcaptionbox{Comparison of the impact of different time accumulation types on the estimator}[\textwidth]{\includegraphics[width=0.9\textwidth]{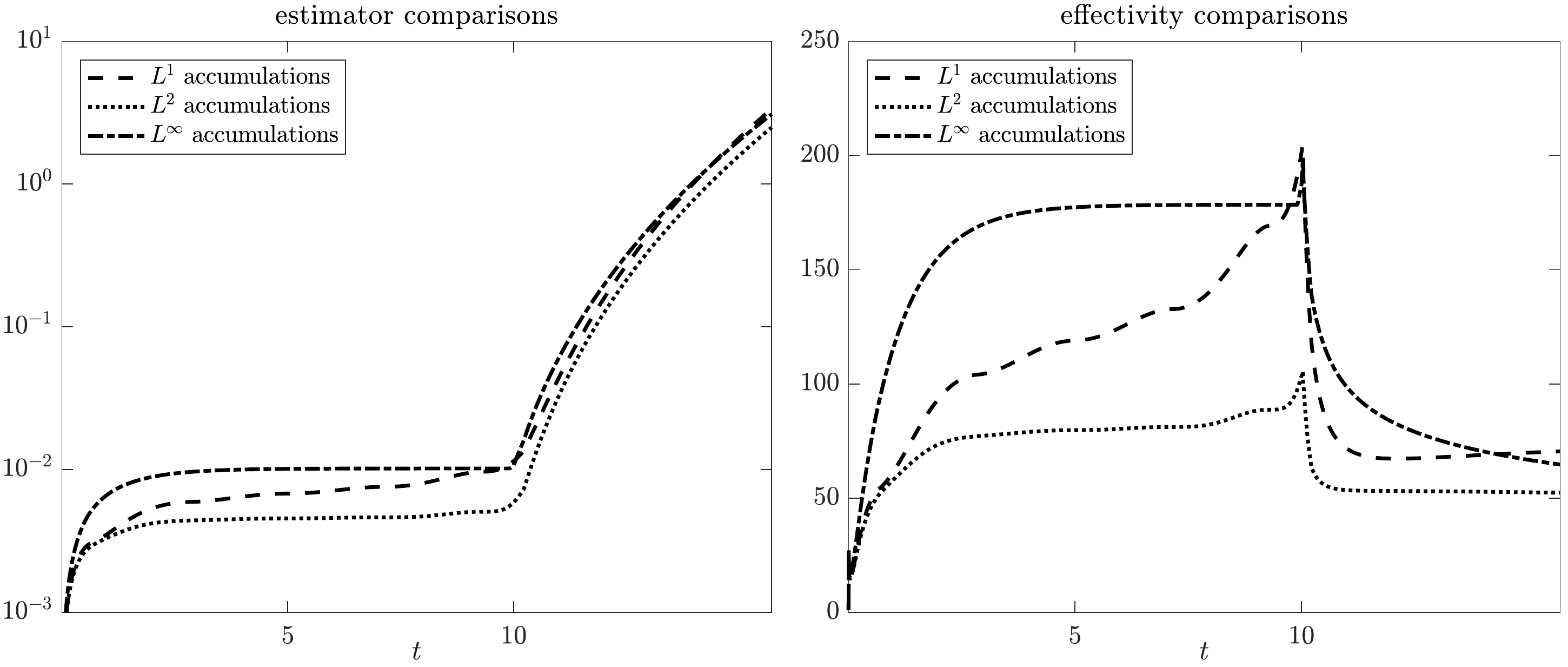}\vspace{-0.75em}}%
  }%
  \caption{Behaviour of the error and estimator for the backward Euler scheme applied to the polynomial example~\eqref{eq:ex:poly} with $\tau \approx h^2$.}
  \label{fig:be:poly}
\end{figure}

\begin{figure}%
  \centering{%
  \subcaptionbox{Behaviour of the error and estimator}[\textwidth]{\includegraphics[width=0.9\textwidth]{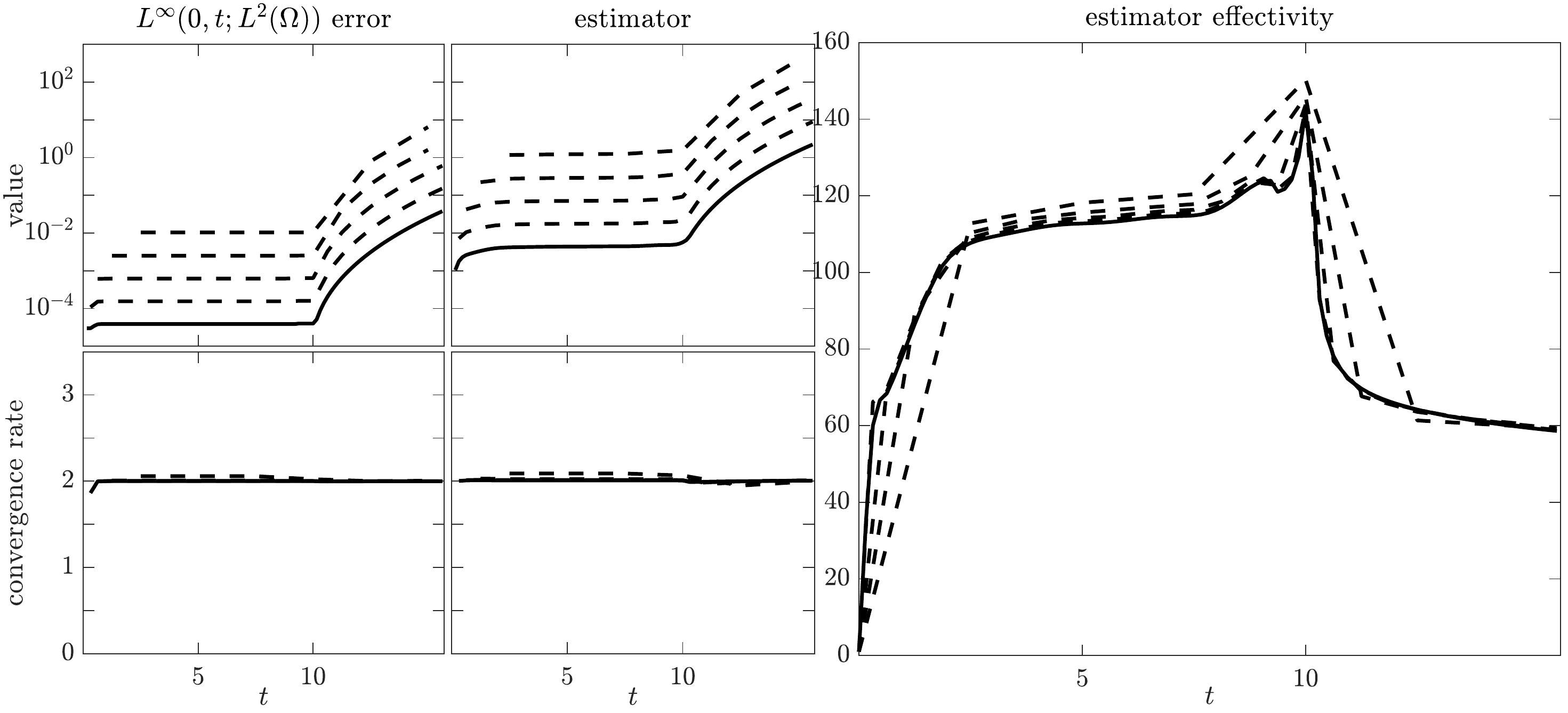}\vspace{-0.75em}}
  \vspace{1em}
  
  \subcaptionbox{Behaviour of the components of the estimator}[\textwidth]{\includegraphics[width=0.9\textwidth]{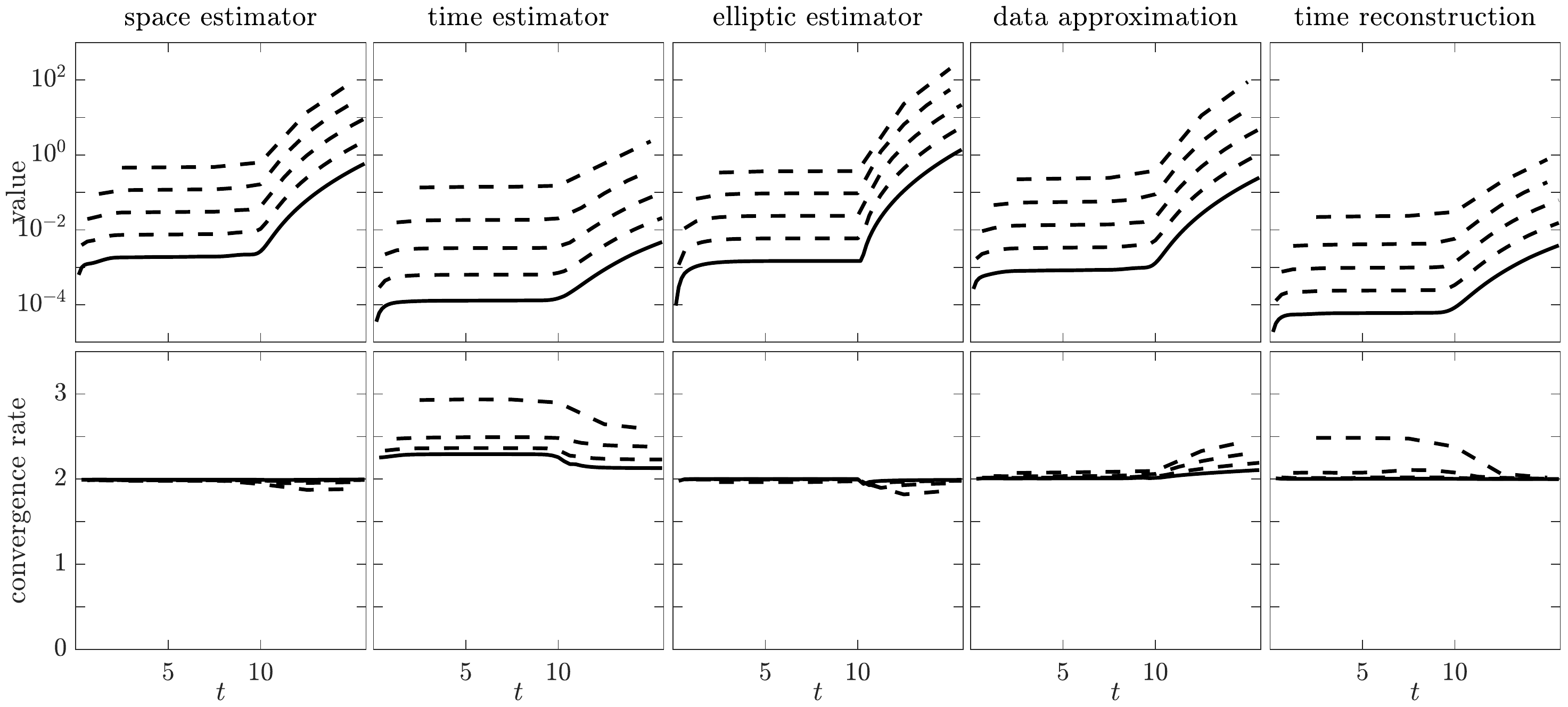}\vspace{-0.75em}}
  \vspace{1em}
  
  \subcaptionbox{Comparison of the impact of different time accumulation types on the estimator}[\textwidth]{\includegraphics[width=0.9\textwidth]{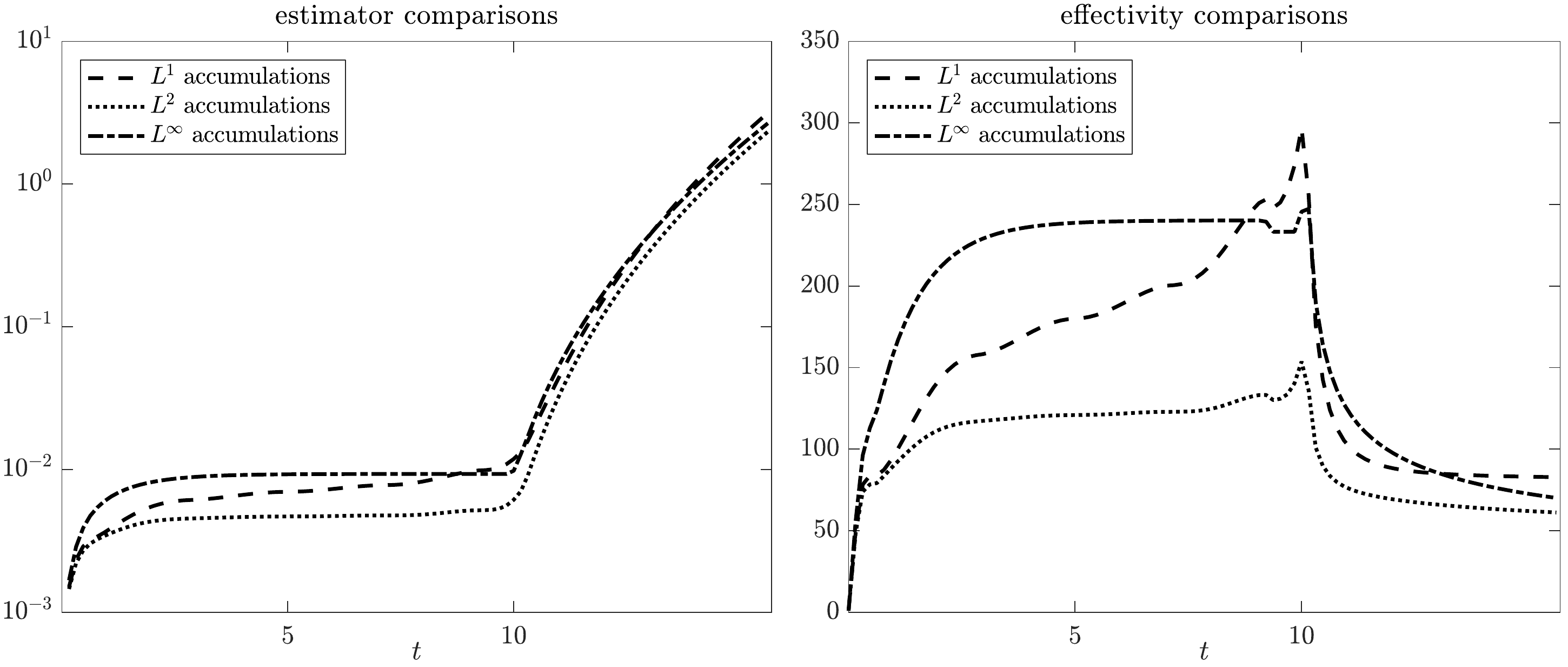}\vspace{-0.75em}}%
  }%
  \caption{Behaviour of the error and estimator for the Crank-Nicolson scheme applied to the polynomial example~\eqref{eq:ex:poly} with $\tau \approx h$.}
  \label{fig:cn:poly}
\end{figure}